\documentclass{amsart} 

\usepackage{amsmath,amssymb,amsthm,ifpdf}
\usepackage{enumerate}
\usepackage{eucal} 
\usepackage[bookmarks,colorlinks=true]{hyperref}
\usepackage{mathabx}
\ifpdf
\usepackage{hyperref}
\else
\usepackage[hypertex]{hyperref}
\fi
\usepackage{color}
\definecolor{red}{rgb}{0.9,0.3,0.3}

\newcommand{\Cyl}[1]{\ensuremath{\ldbrack{#1}\rdbrack}}
\newcommand{\seq}[1]{{\left\langle{#1}\right\rangle}}
 
\newcommand{\uhr}[1]{\! \upharpoonright_{#1}} 
\newcommand{\rest}{\uhr}

\newcommand{\andd}{\,\,\,\&\,\,\,}
\DeclareMathOperator \dom{dom}
\newcommand \conc{\widehat{\phantom{\alpha}}}
\newcommand{\cero}{\mathbf{0}}
\newcommand{\converge}{\!\!\downarrow}
\newcommand{\diverge}{\!\!\uparrow}
\newcommand{\w}{\omega}
\newcommand \s{\sigma}
\renewcommand \leq {\leqslant}
\renewcommand \geq {\geqslant}
\renewcommand \phi{\varphi}
\newcommand{\nbd}{\nobreakdash-\hspace{0pt}}
\newcommand{\comment}[1]{}

\newcommand \fin{\texttt{fin}}
\newcommand \inff{\texttt{inf}}

\newcommand \ZZ{\mathcal Z}
\newcommand \YY{\mathcal Y}
\newcommand \XX{\mathcal X}
\newcommand \WW{\mathcal W}
\newcommand \VV{\mathcal V}
\newcommand \EE{\mathcal E}
\newcommand \CC{\mathcal C}
\newcommand \UU{\mathcal U}
\newcommand \BB{\mathcal B}
\renewcommand \SS{\mathcal S}

\newcommand \leb{\lambda}

\newcommand \DemBLR{\textup{Demuth}_{\textup{BLR}}}

\theoremstyle{plain}
\newtheorem{theorem}{Theorem}[section] 
\newcounter{claimCounter}[theorem]

\newtheorem{lemma}[theorem]{Lemma} 
\newtheorem{cclaim}[theorem]{Claim}
\newtheorem{corollary}[theorem]{Corollary} 
 
\newtheorem{claim}[claimCounter]{Claim}

\theoremstyle{definition}
\newtheorem{definition}[theorem]{Definition}

\theoremstyle{remark}

\title{Strong Jump Traceability and Demuth Randomness}

\author{Noam Greenberg}
\address{School of Mathematics, Statistics and Operations Research, Victoria University of Wellington,
  Wellington, New Zealand}
\email{\href{mailto:Noam.Greenberg@msor.vuw.ac.nz}%
{Noam.Greenberg@msor.vuw.ac.nz}}
\urladdr{\url{http://homepages.mcs.vuw.ac.nz/~greenberg/}}

\author{Daniel D. Turetsky}
\address{School of Mathematics, Statistics and Operations Research, Victoria University of Wellington,
  Wellington, New Zealand}
\email{\href{mailto:dan.turetsky@msor.vuw.ac.nz}{dan.turetsky@msor.vuw.ac.nz}}
\urladdr{\url{http://msor.victoria.ac.nz/Main/DanTuretsky}}

\thanks{Both authors were supported by the Marsden Fund of New Zealand, the second author as a postdoctoral fellow.}

\begin{document}

\begin{abstract}
We solve the covering problem for Demuth randomness, showing that a computably enumerable set is computable from a Demuth random set if and only if it is strongly jump-traceable. We show that on the other hand, the class of sets which form a base for Demuth randomness is a proper subclass of the class of strongly jump-traceable sets. 
\end{abstract}

\maketitle

\section{Introduction}

Hirschfeldt, Nies and Stephan \cite{HirschfeldtNiesStephan:UsingRandomOracles} showed that every computably enumerable (c.e.) set which is computable in an incomplete Martin-L\"of random set is $K$-trivial. The question whether the converse holds is known as the \emph{covering problem} in algorithmic randomness. To date, this problem remains open, and is considered one of the major open problems in the field; see \cite{Miller_Nies:_randomness_open_questions}. 

This question lies at the heart of the study of the relationship between algorithmic randomness and the Turing degrees. The origin of this research programme can be traced back to Ku\v{c}era's \cite{Kucera:86}, in which he showed that every $\Delta^0_2$ Martin-L\"of random set computes a noncomputable c.e.\ set; this allowed him to use the low basis theorem to provide an injury-free solution to Post's problem. In general, researchers study the distribution of the random sets in the Turing degrees, and in particular how these random degrees fit in with other classes of degrees which are examined by classical computability theory, prime among them being the class of c.e.\ degrees. Since incomplete c.e.\ sets cannot compute random sets, the natural question to ask is: which random sets compute which c.e.\ sets? The covering problem is one instance of this question, fixing the notion of randomness to be incomplete Martin-L\"of randomness. A positive solution to the covering problem would give us a new characterisation of $K$-triviality, which is the central \emph{lowness notion} of algorithmic randomness. 

\emph{Strong jump-traceability}, introduced by Figueira, Nies and Stephan in \cite{FigueiraNiesStephan}, is another lowness notion of c.e.\ degrees. A lowness notion defines a class of sets which resemble the computable sets in some way, and thus tells us that they are far from being complete. Like other variants of traceability, strong jump-traceability is a combinatorial notion, defined without reference to prefix-free Kolmogorov complexity or Lebesgue measure, and yet interacts with notions from algorithmic randomness. It resembles $K$-triviality: Cholak, Downey and Greenberg \cite{CholakDowneyGreenberg} showed that the strongly jump-traceable c.e.\ degrees form an ideal, properly contained in the ideal of $K$-trivial degrees. Nies and Greenberg showed \cite{GNies} that like the $K$-trivial sets, c.e.\ strongly jump-traceable sets are characterised as those sets that have nice approximations, obeying the so-called \emph{benign cost functions}. Greenberg, Hirschfeldt and Nies \cite{GHNies} then used this characterisation to show that in some sense, the c.e.\ strongly jump-traceable sets behave more nicely that the $K$-trivial sets, since they have both ``continuous'' and ``discrete'' definitions: a c.e.\ set is strongly jump-traceable if and only if it is computable from \emph{all} superlow Martin-L\"of random sets (and in fact, if and only if it is computable from all superhigh Martin-L\"of random sets). This was the first instance of a definition of a class of c.e.\ degrees using their interaction with random sets. 

After Greenberg \cite{G} constructed a $\Delta^0_2$ Martin-L\"of random set which only computes strongly jump-traceable c.e.\ sets, Ku\v{c}era and Nies \cite{KuceraNies} showed that any c.e. set computable from any \emph{Demuth} random set is strongly jump-traceable. Demuth randomness was introduced by Demuth \cite{Demuth_classes_of_arithmetical,Demuth:88} in order to study differentiability of constructive functions; he showed that every constructive function satisfies the Denjoy alternative at any Demuth random real (the converse is still open, but it is known that some strengthening of Martin-L\"of randomness is required; see for example \cite{Demuth:Denjoy}). Demuth randomness is a strengthening of Martin-L\"of randomness which has some nice properties which resemble Cohen 1-genericity: it implies generalised lowness (and so in particular incompleteness), but unlike weak 2-randomness is compatible with being $\Delta^0_2$. 

Ku\v{c}era's and Nies's result, much like the Hirschfeldt-Nies-Stephan result mentioned above, raises the question of whether the converse holds. This is the variant of the covering problem for Demuth randomness. In this paper, we provide a positive solution to this problem.

\begin{theorem}\label{thm:SJT Below Demuth}
	A c.e.\ set is strongly jump-traceable if and only if it is computable from some Demuth random set. 
\end{theorem}

The proof of Theorem \ref{thm:SJT Below Demuth} is involved, combining novel techniques with the \emph{box-promotion} method use in the investigation of strongly jump-traceable sets. This is the first example using the full power of strong jump-traceability, rather than an approximation in the form of $h$-jump-traceability for some sufficiently slow growing order function $h$. A general argument in the style of \cite{GHNies} is impossible here, since no $\Delta^0_2$ Demuth random set computes all strongly jump-traceable c.e.\ sets. The Demuth random set constructed computing a given c.e., strongly jump-traceable set is $\Delta^0_3$; it remains open whether it can be made $\Delta^0_2$. 

\

Being a \emph{base} for a notion of randomness is a lowness notion emanating from the interplay of randomness and Turing reducibility. If $\mathcal R$ is a relativisable class of randomness, then we say that a set $A$ is a base for $\mathcal R$ if there is some $X\in {\mathcal R}^A$ which computes $A$. That is, $A$ resembles the computable sets in that the cone of degrees above $A$, while being null, nevertheless intersects an $A$-definable conull class, namely ${\mathcal R}^A$. The robustness of the class of $K$-trivial degrees is witnessed by its coincidence with the class of bases for Martin-L\"of randomness (Hirchfeldt, Nies and Stephan \cite{HirschfeldtNiesStephan:UsingRandomOracles}). Nies \cite{NiesDemuth} showed that every base for Demuth randomness is strongly jump-traceable, and asked if the converse holds. That is, whether Theorem \ref{thm:SJT Below Demuth} can be improved to produce not merely a Demuth random set computing a given strongly jump-traceable set $A$, but indeed a Demuth${^A}$ random set computing $A$. We show that the converse fails, even when restricted to c.e.\ sets. 

\begin{theorem}\label{thm:SJT Not Base for Demuth Randomness}
	There is a strongly jump-traceable c.e.\ set which is not a base for Demuth randomness.
\end{theorem}

Nies showed that the class of c.e.\ bases for Demuth randomness properly contains a sub-ideal of the c.e.\ jump-traceable sets, namely those c.e.\ sets computable from every $\w^2$-computably approximable Martin-L\"of random sets. Thus, the collection of bases for Demuth randomness forms a new class, about which we know close to nothing. For example, it is not clear if it induces an ideal in the Turing degrees. 

It is easy to prove that every $K$-trivial set is a base for Martin-L\"of randomness, once it is shown that $K$-triviality implies \emph{lowness} for Martin-L\"of randomness. That is, if $A$ is $K$-trivial, then every Martin-L\"of random set is Martin-L\"of random relative to $A$. By the Ku\v{c}era-G\'{a}cs theorem, $A$ is computable from a Martin-L\"of random set $Z$ (indeed every $K$-trivial set is $\Delta^0_2$, so $A$ is computable from Chaitin's $\Omega$), and so $Z$ witnesses that $A$ is a base for Martin-L\"of randomness. A na\"ive attempt to show that every c.e., strongly jump-traceable set is a base for Demuth randomness would start by utilising Theorem \ref{thm:SJT Below Demuth} as an analogue to the Ku\v{c}era-G\'{a}cs theorem, and then go on to show that every strongly jump-traceable set is low for Demuth randomness. Unfortunately, the latter fails. Indeed, Downey and Ng \cite{DowneyNg:LowDemuth} showed that lowness for Demuth randomness implied hyperimmune-freeness, whereas Downey and Greenberg \cite{DGSJT2} showed that every strong jump-traceable set is $\Delta^0_2$, and so the only strongly jump-traceable sets that are low for Demuth are the computable ones. 

The proof of Theorem \ref{thm:SJT Not Base for Demuth Randomness} relies on the fact that the full relativisation of Demuth randomness to an oracle $A$ allows for an $A$-computable bound on the number of mind-changes for the value of the function giving the index for components of a Demuth test. This prompts the definition of a related notion of randomness, $\DemBLR$ randomness, which is a partial relativisation of Demuth randomness, prohibiting this increased bound on the number of mind-changes. This notion is studied in \cite{DemuthBLRpaper}, where in particular it is shown that every strongly jump-traceable set is {low} for $\DemBLR$-randomness. This allows us to resuscitate the na\"ive plan from the previous paragraph and conclude: 

\begin{corollary}\label{cor_SJT_base}
	A c.e.\ set is strongly jump-traceable if and only if it is a base for $\DemBLR$-randomness. 
\end{corollary}

We prove Theorem \ref{thm:SJT Below Demuth} in Section \ref{sec:main}, and Theorem \ref{thm:SJT Not Base for Demuth Randomness} in Section \ref{sec:base}. 

\

\section{Definitions of Demuth randomness and other notions}

We first define strong jump-traceability.

\begin{definition}\label{def:sjt} \
	\begin{enumerate}
		\item An \emph{order function} is a computable, nondecreasing and unbounded function $h\colon \w\to \w\setminus\{0\}$. 
		
		\item A \emph{c.e.\ trace} is a uniformly c.e.\ sequence of finite sets. A c.e.\ trace $\seq{T_x}_{x<\w}$ \emph{traces} a partial function $\psi\colon \w\to \w$ if for all $x\in \dom \psi$, $\psi(x)\in T_x$. 
		
		\item If $h$ is an order function, then an $h$-trace is a c.e.\ trace $\seq{T_x}$ such that for all $x<\w$, $|T_x|\le h(x)$. 
		
		\item A set $A$ is \emph{strongly jump-traceable} if for every order function $h$, every $A$-partial computable function $\psi$ is traced by an $h$-trace. 
	\end{enumerate}
\end{definition}

Next, we discuss notation for subsets of Cantor space $2^\w$, and define Demuth randomness. 

\begin{definition}\label{def_omega_ca}
	A sequence of functions $\seq{f_s}_{s<\w}$ is an \emph{approximation} of a function $f\colon \w\to \w$ if for all $n$, for all but finitely many $s$, $f_s(n) = f(n)$. We often write $f(n,s)$ for $f_s(n)$. A \emph{computable approximation} is a uniformly computable sequence which is an approximation. Shoenfield's limit lemma says that a function has a computable approximation if and only if it is computable from $\cero'$. 
	
	If $\seq{f_s}$ is an approximation, then the associated \emph{mind-change} function $m_{\seq{f_s}}$ is defined by
	\[ m_{\seq{f_s}}(n)  = \# \left\{ s\,:\, f_{s+1}(n) \ne f_s(n) \right\} .\]
	A computable approximation $\seq{f_s}$ is an \emph{$\w$-computable approximation} if $m_{\seq{f_s}}$ is bounded by a computable function. A function is \emph{$\w$-computably approximable} (or \emph{$\w$-c.a.}) if it has an $\w$-computable approximation. 	
\end{definition}

\begin{definition}\label{def_effectively_open}
	For a finite binary string $\s\in 2^{<\w}$, we let $\Cyl{\s}$, the clopen subset defined by $\s$, be the collection of reals $X\in 2^\w$ which extend $\s$. If $W$ is a set of strings, then 
	\[ \Cyl{W} = \bigcup_{\s\in W} \Cyl{\s}\]
	is the open (or $\mathbf{\Sigma}^0_1$) subset of $2^\w$ defined by $W$. If $W$ is c.e., then $\Cyl{W}$ is called \emph{effectively open} (or $\Sigma^0_1$). By compactness, a subset $\VV$ of Cantor space is clopen if and only if $\VV = \Cyl{D}$ for some finite subset $D$ of $2^{<\w}$. 
	
	If $\WW = \Cyl{W}$ and $\seq{W_s}$ is an effective enumeration of the c.e.\ set $W$, then we often write $\WW_s$ for $\Cyl{W_s}$. We call $\seq{\WW_s}$ an \emph{effective enumeration} of $\WW$.
\end{definition}

We interrupt the stream of definitions to remark that we will be using Lachlan's notation \cite{Lachlan_nonbounding} of appending the stage in square brackets to a complicated expression to indicate that every element of the expression is intended to be evaluated at that stage. For example, if $\seq{f_s}$ is a computable approximation of a function $f$, and $\VV_s$ is an effective enumeration of $\VV$, then we write $\VV \cup \Cyl{W_{f(n)}}\,[s]$ rather than $\VV_s \cup \Cyl{W_{f_s(n),s}}$.

\begin{definition}\label{def_open_test}
	A \emph{test} is a sequence $\seq{\VV_n}_{n<\w}$ of open subsets of Cantor space $2^\w$ such that for all $n$, $\leb(\VV_n)\le 2^{-n}$; here $\leb$ denotes the fair coin measure on Cantor space. We say that a set $X\in 2^\w$ \emph{passes} the test $\seq{\VV_n}$ if $X\in \VV_n$ for only finitely many $n$. Otherwise, the set $X$ \emph{fails} the test. The collection of sets which fail a test is a null class. 
	
	A test $\seq{\VV_n}$ is \emph{effectively open} if each $\VV_n$ is an effectively open subset of Cantor space. If $\seq{\VV_n}$ is effectively open, then an \emph{index function} for $\seq{\VV_n}$ is a function $f\colon \w\to \w$ such that for all $n$, $\VV_n = \Cyl{W_{f(n)}}$; here $\seq{W_e}$ is an effective list of all c.e.\ sets.  Thus, for example, an effectively open test is a Martin-L\"of test if it has a computable index function. A \emph{Demuth test} is an effectively open test which has an $\w$-c.a.\ index function. A set $X\in 2^\w$ is \emph{Demuth random} if it passes all Demuth tests. 	
\end{definition}

Rather than working with Demuth tests, it will be convenient to work with a more restrictive (yet equally powerful) notion of tests.

\begin{definition}\label{def_Kutrz_test}
	A test $\seq{\VV_n}$ is \emph{clopen} if each $\VV_n$ is a clopen subset of $2^\w$. If $\seq{\VV_n}$ is a clopen test, then a \emph{clopen index function} for $\seq{\VV_n}$ is a function $f\colon \w\to \w$ such that for all $n$, $\VV_n = \Cyl{D_{f(n)}}$; here $\seq{D_e}$ is an effective list of all finite sets of strings. Thus, for example, a Kurtz test is a clopen test which has a computable clopen index function. A \emph{Demuth clopen test} is a clopen test which has an $\w$-c.a.\ clopen index function. 
\end{definition}

The following lemma is implicit in \cite{HoelzlKraelingStephanWu}. We give a proof for completeness. 

\begin{lemma}\label{lem:clopen tests}
A set $X\in 2^\w$ is Demuth random if and only if it passes every clopen Demuth test. 
\end{lemma}

\begin{proof}
Every clopen Demuth test is a Demuth test. Hence every Demuth random set passes every clopen Demuth test. 

For the converse, we show that for any Demuth test $\seq{\VV_n}$ there is a clopen Demuth test $\seq{\UU_n}$ such that every set which fails the test $\seq{\VV_n}$ also fails the test $\seq{\UU_n}$. Let $\seq{\VV_n}$ be a Demuth test.  

The idea is to copy $\bigcup_n \VV_{n}$ into various $\UU_n$'s in discrete steps. For each $\UU_n$, we set a threshold $\epsilon(n)$. We then copy $\VV_{n}$ into $\UU_{n}$ only at stages at which the measure of $\VV_n$ passes some integer multiple of $\epsilon(n)$. At other stages, the part of $\VV_{{n}}$ which hasn't yet been copied to $\UU_n$ is split up and copied to $\UU_m$ for various $m>n$, depending on the measure of that part and its relation to the thresholds $\epsilon(m)$. At a later stage, if the measure of $\VV_{n}$ crosses another integer multiple of $\epsilon(n)$, we recall that part of $\VV_{n}$ which has been passed to $\UU_m$ for $m>n$, and copy it to $\UU_n$. Because $\epsilon(n)$ is fixed, $\UU_n$ is changed only a finite number of times, and so $\UU_n$ is clopen. 

Actually, this description is not quite correct, because we can set $\epsilon(n)$ to be greater than $\leb(\VV_n)$, so $\leb(\VV_{n,s})$ never crosses an integer multiple of $\epsilon(n)$. What we in fact track, when defining $\UU_n$, is the total measure of the parts of $\VV_{k,s}$ for $k<n$ which are passed down to $\UU_n$. 

To assist with the construction, we will define auxiliary clopen sets~$\seq{\SS_{n,s}}$. These consist of the measure passed on to $\UU_{n}$ by $\UU_{n-1}$, together with $\VV_{n,s}$. Let $f$ be an $\w$-c.a.\ index function for $\seq{\VV_n}$, and let $\seq{f_s}$ be an $\w$-computable approximation for $f$. We let $\VV_{n,s} = \Cyl{W_{f_s(n),s}}$. Since each set $W_{e,s}$ is finite, each set $\VV_{n,s}$ is clopen (in fact, a canonical index $d$ such that $W_{f_s(n),s} = D_d$ can be obtained effectively from $n$ and $s$). We may assume that for all $n$ and $s$, $\leb(\VV_{n,s}) \le 2^{-n}$, and that for all $s$, for all $n\ge s$, $\VV_{n,s}=\emptyset$. 

For all $n$, we let $\epsilon(n) = 2^{-n}$.

\subsubsection*{Construction}

At stage 0, we let $\UU_{n,0} = \SS_{n,0} = \emptyset$ for all $n$. At stage~$s>0$ we define~$\SS_{n,s}$ and $\UU_{n,s}$ for all~$n$ by recursion on $n$. We first let $\SS_{0,s} = \emptyset$. Let $n<\w$, and suppose that $\SS_{n,s}$ is already defined. 

If 
\[ \leb\left(\SS_{n,s}\setminus \UU_{n,s-1}\right) > \epsilon(n),\]
then $\UU_{n}$ needs to change; we let $\UU_{n,s} = \SS_{n,s}$, and for all $m> n$ we let $\UU_{m,s} = \SS_{m,s} = \emptyset$. Otherwise, we let $\UU_{n,s} = \UU_{n,s-1}$, let 
\[ \SS_{n+1,s} =  \VV_{n,s} \cup \left(\SS_{n,s}  \setminus \UU_{n,s} \right) ,\]
and proceed to define $\UU_{n+1,s}$.

\subsubsection*{Verification}

For all $n$ and $s$, 
\[ \leb\left(\SS_{n+1,s}\right)\le \epsilon(n) + \leb\left(\VV_{n,s}\right) \le 2^{-n+1}.\] 
Hence, for all $n$ and $s$, $\leb\left( \UU_{n,s}\right)\le 2^{-n+1}$. 

We also see that even though this was not required for the construction to be computable, every stage of the construction is in fact finite. We show, by induction on $s$, that for almost all $n$, $\SS_{n,s} = \UU_{n,s}= \emptyset$. Suppose this holds at stage $s-1$. Suppose, for contradiction, that for infinitely many $n$ we have $\SS_{n,s}\ne \emptyset$; so no $\UU_n$ ``acts'' at stage $s$, and for all $n$ we have $\UU_{n,s}= \UU_{n,s-1}$. Since for all $n\ge s$, $\VV_{n,s}= \emptyset$, and for almost all $n$, $\UU_{n,s-1}$ is empty, for almost all $n$, we have $\SS_{n+1,s} = \SS_{n,s}$; so we are assuming that this stable set is nonempty, and hence has positive measure. Since $\epsilon(n)\to 0$, there is some $n$ such that $\epsilon(n)<\leb(\SS_{n,s})$, and this $n$ would act at stage $s$ and set $\SS_{m,s} = \emptyset$ for all $m>n$, yielding a contradiction. Hence, for almost all $n$, $\SS_{n,s}=\emptyset$; this implies that for almost all $n$, $\UU_{n,s}= \UU_{n,s-1} = \emptyset$. 

\

There is a uniformly computable sequence $\seq{h_s}$ of functions such that for all $n$ and $s$, $\UU_{n,s} = \Cyl{D_{h_s(n)}}$. 

\begin{claim}
The sequence $\seq{h_s}$ is an $\w$-computable approximation.
\end{claim}

\begin{proof}
Fix $n<\w$. Let $s_0>0$ be a stage $s$ such that $\UU_{n,s}\ne \UU_{n,s-1}$. Suppose further that for all $m<n$, $\UU_{m,s_0} = \UU_{m,s_0-1}$. Hence at stage $s_0$ we define $\UU_{n,s_0} = \SS_{n,s_0}$, but for all $m<n$, we have $\SS_{m+1} = \VV_{m}\cup\left( \SS_{m}\setminus \UU_{m}\right)\,\,[s_0]$.

Suppose that there is some stage $s>s_0$ such that $\UU_{n,s}\ne \UU_{n,s-1}$; let $s_1$ be the least such stage. We claim that there is some $m<n$ for which one of the following holds: 
\begin{enumerate}
	\item $\UU_{m,s_1}\ne \UU_{m,s_1-1}$.
	\item There is some $s\in [s_0,s_1)$ such that $f_{s}(m) \ne f_{s-1}(m)$. 
	\item $\leb \left(\VV_{m,s_1} \setminus \VV_{m,s_0}\right) > \epsilon(n)/n$. 
\end{enumerate}	
Suppose that (1) and (2) do not hold. To show that (3) holds, we show that in this case,
	\[ 	\SS_{n,s_1} \setminus \UU_{n,s_0} \subseteq \bigcup_{m < n} \left(\VV_{m,s_1}\setminus \VV_{m,s_0}\right); \]
(3) then follows from the fact that the minimality of $s_1$ ensures that $\UU_{n,s_1-1} = \UU_{n,s_0}$, and from the fact that $\leb\left(\SS_{n,s_1}\setminus \UU_{n,s_1-1} \right)> \epsilon(n)$. To verify the containment, let $X\in \SS_{n,s_1}\setminus \UU_{n,s_0}$. Since (1) does not hold, for all $m<n$, $\SS_{m+1} = \VV_{m}\cup\left( \SS_{m}\setminus \UU_{m}\right)\,\,[s_1]$. By minimality of $s_1$, and since (2) does not hold, for all $m<n$, $\UU_{m,s_1} = \UU_{m,s_0}$. Since $\SS_{n}\subseteq \bigcup_{m<n}\VV_{m}\,\,[s_1]$, there is some $m<n$ such that $X\in \VV_{m,s_1}$; pick $m^*$ to be the greatest such $m$. Then $X\in \SS_{n,s_1}$ implies that for all $m\in (m^*,n)$, $X\notin \UU_{m,s_1}$. Now if $X\in \VV_{m^*,s_0}$, then the fact that $X\notin \VV_{m,s_0}$ for all $m\in (m^*,n)$ would imply that $X\in \SS_{n,s_0}$ and so $X\in \UU_{n,s_0}$. Hence $X\in \VV_{m^*,s_1}\setminus \VV_{m^*,s_0}$ as required. 

This analysis allows us to recursively define a bound $k(n)$ for $m_{\seq{h_s}}$. Let $g$ be a bound for $m_{\seq{f_s}}$. We can let $k(0)=0$, as $\UU_{0,s}= \emptyset$ for all $s$. If $k(m)$ is defined for all $m<n$, then we can let 
\[ k(n) = \left( \sum_{m<n} k(m) \right) \cdot \left( \sum_{m<n} g(m) \right) \cdot \frac{n^2}{\epsilon(n)}.\]

Note that $k$ depends only on $g$ and not on $f$. 
\end{proof}

Let $h = \lim_s h(s)$; for $n<\w$, let $\UU_n = \Cyl{D_{h(n)}} = \lim_s \UU_{n,s}$. Hence $\seq{\UU_{n+1}}_{n<\w}$ is a clopen Demuth test. It remains to see that every set $X\in 2^\w$ which fails the test $\seq{\VV_n}$ also fails the test $\seq{\UU_{n+1}}$. This follows from the following claim. 

\begin{claim}
For all $n<\w$, 
\[ \VV_{n} \subseteq \bigcup_{m>n} \UU_m.\]
\end{claim}

\begin{proof}
	Let $n<\w$ and let $X\in \VV_n$. Let $s_0$ be a stage sufficiently late so that for all $s\ge s_0$, $X\in \VV_{n,s}$, and so that for all $s\ge s_0$, $\UU_{n,s} = \UU_{n,s-1}$. Hence for all $s\ge s_0$, we let  $\SS_{n+1} = \VV_{n}\cup\left( \SS_{n}\setminus \UU_{n}\right)\,\,[s]$, and so for all $s\ge s_0$, $X\in \SS_{n+1,s}$. 
	
	First, we see that for all $s\ge s_0$ there is some $m>n$ such that $X\in \UU_{m,s}$. We saw above that there is some $m>n$ such that $\SS_{m,s}= \emptyset$; so there is some $m>n$ such that $X\in \SS_{m,s}\setminus \SS_{m+1,s}$. Either $\UU_{m,s} = \SS_{m,s}$, or $\SS_{m+1,s} \supseteq \SS_{m,s}\setminus \UU_{m,s}$; in either case, $X\in \UU_{m,s}$. For $s\ge s_0$, let $m(s)$ be the least $m>n$ such that $X\in \UU_{m,s}$. 	
	
	The function $m(s)$ is nonincreasing. To see this, let $s>s_0$, and suppose that $m(s)\ne m(s-1)$. Suppose, for contradiction, that $m(s)>m(s-1)$. Then for all $m\in (n,k]$, we have $X\notin \UU_{m,s}$. Since $X\in \UU_{k,s-1}$, we have $\UU_{k,s}\ne \UU_{k,s-1}$. This implies that for all $m>k$, $\UU_{m,s}=\emptyset$, contradicting $m(s)>m$ and $X\in \UU_{m(s),s}$. 
	
	Hence $m = \lim_s m(s)$ exists, and for almost all $s$, $X\in \UU_{m,s}$. Hence $X\in \UU_m$, as required. 	
\end{proof}
\end{proof}

We would like to draw the reader's attention to certain terminology that was used in the last proof, and will be used throughout the paper. In a couple of instances, the word ``measure'' meant ``a nonempty clopen subset of Cantor space'', as in ``the measure passed on to $\UU_{n}$ by $\UU_{n-1}$''. This incorrect usage of the word ``measure'' makes for smoother sentences, but also emphasises that we often don't quite care which particular nonempty clopen sets we are dealing with, but rather care about its measure. 

\

The keep future calculations smoother, we employ quick tests. 

\begin{definition}\label{def_quick_test}
	A test $\seq{\VV_n}$ is \emph{quick} if for all $n$, $\leb(\VV_n) \le 2^{-2n}$. 
\end{definition}

\begin{lemma}
A set $X$ is Demuth random if and only if it passes every quick clopen Demuth test.
\end{lemma}

\begin{proof}
Let $\seq{\VV_n}$ be a clopen Demuth test. For $n<\w$, let $\UU_n = \VV_{2n+1} \cup \VV_{2n+2}$. Then 
\[ \leb(\UU_n) \le 2^{-2n+1} + 2^{-2n+2} < 2^{-2n},\]	
so $\seq{\UU_n}$ is a quick test, and it is easy to see that $\seq{\UU_n}$ is a clopen Demuth test. If $X$ fails $\seq{\VV_n}$ then it fails $\seq{\UU_n}$. 	
\end{proof}

In general, it can be shown that if $\seq{q_n}$ is a computable, nonincreasing sequence of rational numbers, and $\sum_n q_n$ converges to a computable real number, then a set is Demuth random if and only if it passes all clopen Demuth tests $\seq{\VV_n}$ satisfying  $\leb(\VV_n)\le q_n$ for all $n$. We do not require this generality in this paper.

\

We fix an enumeration of quick clopen Demuth test. Using a uniform enumeration of all $\w$-c.a.\ functions, we fix an effective list $\seq{\VV_{n,s}^e}$ of clopen sets (that is, canonical indices are given effectively), and an effective list $\seq{g^e}$ of partial computable functions, such that:
\begin{itemize}
	\item For all~$n, e$ and~$s$, $\leb(\VV_{n,s}^e) \leq 2^{-2n}$;
	\item For all~$n$ and~$e$, if $n\in \dom g^e$, then $\# \left\{ s\,:\, \VV_{n,s}^e \ne \VV_{n,s+1}^e\right\} \le g^e(n)$;
	\item For all~$n$ and~$e$, if $n\notin \dom g^e$, then for all $s$, $\VV_{n,s}^e = \emptyset$;
	\item For all~$e$, the domain of $g^e$ is an initial segment of~$\w$;
	\item For every quick clopen Demuth test $\seq{\VV_n}$, there is an~$e$ such that~$g^e$ is total and $\VV_n = \VV_n^e$, where $\VV_n^e = \lim_s \VV_{n,s}^e$; and
	\item $g^0$ is total.
\end{itemize}

\section{Proof of Theorem \ref{thm:SJT Below Demuth}} \label{sec:main}

By the Ku\v{c}era-Nies result from \cite{KuceraNies}, it is sufficient to show that every strongly jump-traceable c.e.\ set is computable from some Demuth random set. Let $A$ be a strongly jump-traceable c.e.\ set. Let $\seq{A_s}$ be an effective enumeration of~$A$. 

We want to construct a Demuth random set that computes $A$. To do so, we enumerate a Turing functional $\Gamma$. A typical axiom, enumerated into~$\Gamma$ at a stage~$s$ of the construction, will map a clopen subset $\CC$ of Cantor space to some initial segment of $A_s$. At the end we let, for $X\in 2^\w$, 
\[
\Gamma^X = \bigcup \Gamma(\CC) \,\,\,\Cyl{X\in \CC}.
\]
Because $A$ is c.e., to keep $\Gamma$ consistent it is sufficient (and necessary) to ensure that if $\CC$ is added to the domain of $\Gamma$ at stage~$s$, then $\CC$ is disjoint from the {\em error set}:
\[
\EE_s = \bigcup \CC \,\,\,\Cyl{\CC\in \dom \Gamma \andd \Gamma(\CC)\nsubset A_s}.
\]
Our aim is to construct $\Gamma$ so that there is some $X$ such that $\Gamma^X = A$, and $X$ passes every Demuth test. There are therefore three tasks at hand:
\begin{itemize}
	\item Ensure that $X\notin \EE$, where
	\[ \EE = \bigcup_s \EE_s = \bigcup \CC \,\,\,\Cyl{\CC\in \dom \Gamma \andd \Gamma(\CC)\nsubset A} ;\]
	\item Ensure that for all $k$ there is some $\CC\in \dom \Gamma$ such that $X\in \CC$ and $|\Gamma(\CC)| \ge k$;
	\item Ensure that for all $e$ such that $g^e$ is total, there is some $n_e$ such that for all $n\ge n_e$, $X\notin \VV^e_{n}$. 
\end{itemize}

\subsection{Towards a full strategy}

We begin by illustrating simplified approaches to the construction, what goes wrong, and the added complexity needed to address these issues.

For every~$k$, we would like to have some clopen set~$\UU^k$ with $\Gamma(\UU^k) = A\uhr{k}$.  We would also like these to be nested, so that~$\UU^{k+1} \subseteq \UU^k$.  Then by compactness there is some $X \in \bigcap_k \UU^k$. For such an~$X$ we would have $\Gamma^X = A$.

The simplest approach to constructing these is to simply select some clopen set~$\UU^k$ and define $\Gamma(\UU^k) = A_k\uhr{k}$.  Of course, assuming~$A$ is non-computable, there will be~$k$ such that $A_k\uhr{k} \neq A\uhr{k}$.  When we see~$A\uhr{k}$ change, all the measure in~$\UU^k$ becomes bad (it enters~$\EE$), so we need to select new measure from~$\UU^{k-1}$ and use that to redefine~$\UU^k$ (defining $\Gamma(\UU^k) = A_s\uhr{k}$ for this new~$\UU^k$).  Since~$A\uhr{k}$ can only change~$k$ many times, choosing the sizes of the~$\UU^k$ appropriately will guarantee that there is always sufficient measure.

Of course, the above approach makes no effort to ensure that~$X$ is Demuth random.  So suppose~$\seq{\VV_n}$ were some Demuth test; we wish to ensure that~$X$ is not covered by (i.e.\ passes)~$\seq{\VV_n}$.  The easiest approach would be to assign some~$k$ the task of avoiding the test, and whenever any~$\VV_{n,s}$ covers part of~$\UU^k$, remove~$\VV_{n,s} \cap \UU^k$ from~$\UU^k$ and take replacement measure from~$\UU^{k-1}$.  Of course, if~$n$ is small compared to~$k$, it might be that $\UU^{k-1} \subseteq \VV_{n,s}$, so there would be no good replacement measure to take.  This can be solved by choosing an appropriately large value~$n_k$ and only considering~$\VV_{n,s}$ with $n \geq n_k$.

There are still problems with this approach, however.  First, because~$\UU^k$ is trying to avoid infinitely many~$\VV_n$, it will never settle; there will always be an~$n$ for which~$\VV_n$ is still ``moving'' --- the $\w$-c.a.\ approximation function is still changing.  This~$\VV_n$ will cause measure to move out of~$\UU^k$, and so~$\UU^k$ will always be changing.  Thus the limit (however we choose to define that) will not be a closed set, and so the compactness argument above will fail.

Our solution here is the same as in Lemma \ref{lem:clopen tests}: we only change~$\UU^k$ when a critical amount of ``badness'' has built up.  In particular, we only remove measure from~$\UU^k$ when the amount that needs to be replaced is at least $1/4$ the total measure of~$\UU^k$.  Because the~$\VV_n$ shrink quickly, there will be some~$m$ such that only those~$\VV_n$ with $n_k \leq n \leq m$ need be considered; the total combined measure of the~$\VV_n$ for $n > m$ cannot possible be enough to trigger a change in~$\UU^k$.  Once the approximation function has settled for $n \leq m$,~$\UU^k$ will have settled.  So~$\UU^k$ will be closed (actually clopen) as desired.

However, there will be some set $\WW^k = \UU^k \cap \bigcup_{n_k \leq n} \VV_n$ of reals in~$\UU^k$ which may be covered by the test.  This will be an open set (not necessarily effectively so) of measure at most $1/4$ the measure of~$\UU^k$, so $\UU^k - \WW^k$ will be a closed, nonempty set.  So we use $\UU^k - \WW^k$ in place of~$\UU^k$ in the compactness argument above.

Of course, now we need to worry about the sequence being nested.  $\WW^k$ may only be a fraction of the size of~$\UU^k$, but it could be that $\UU^{k+1} \subseteq \WW^k$.  So~$\UU^{k+1}$ will need to avoid~$\WW^k$ in addition to avoiding whatever test~$\seq{\VV_n'}$ it is assigned.  Again,~$\UU^{k+1}$ only replaces measure when the amount needing replacement is at least $1/4$ its total measure.  Now, however, it concerns itself not only with replacing measure covered by its test, but also with replacing measure covered by~$\WW^k$.  Since~$\WW^k$ is small compared to~$\UU^k$, and because~$\UU^{k+1}$ chooses a large~$n_{k+1}$ for its own test, the amount covered at any one time is small, so~$\UU^{k+1}$ should always be able to find good replacement measure in~$\UU^k$ when it needs it.

Now we revisit the first part of our construction, ensuring that $\Gamma(\UU^k) = A\uhr{k}$, and consider how it interacts with this new process.  Suppose~$\UU^k$ is assigned the task of avoiding~$\seq{\VV_n}$, and~$\VV_n$ is some component of the test such that~$\leb(\VV_n)$ is smaller than $1/4$ the measure of~$\UU^k$, but larger than $1/4$ the measure of~$\UU^{k+1}$.  So, on its own,~$\VV_n$ is enough to cause~$\UU^{k+1}$ to change, but not enough to cause~$\UU^k$ to change.  For the moment we ignore the actions of any other components of the test.

At stage~$s$, $\Gamma(\UU^k) = A_s\uhr{k}$ and $\Gamma(\UU^{k+1}) = A_s\uhr{k+1}$.  Suppose that at this stage,~$\VV_n$ changes to cover measure in~$\UU^{k+1}$.  Then~$\UU^{k+1}$ will remove that measure and seek replacement measure from~$\UU^k$.  However, the~$\Gamma$-computation for~$A_s\uhr{k+1}$ still exists on the removed measure.  If~$\VV_n$ changes again to cover new measure in~$\UU^{k+1}$, then~$\UU^{k+1}$ will again seek new measure.  There is a bound on the number of times~$\VV_n$ can change, but it could potentially be large with respect to $\leb(\UU^k)/\leb(\UU^{k+1})$.  So in this fashion,~$\UU^k$ could be filled with $\Gamma$-computations for~$A_s\uhr{k+1}$.  All of this measure is being ``risked'': suppose that after this happens,~$k$ enters~$A$.  Then all of the measure in~$\UU^k$ is bad, since it miscomputes~$A$, but~$A\uhr{k}$ has not changed, so~$\UU^k$ does not believe it should seek replacement measure.

This happened because~$\UU^{k+1}$ was wasteful when it sought out new measure.  The first time~$\VV_n$ covered part of it, it removed some measure and took new measure.  The next time~$\VV_n$ moves, if~$A\uhr{k+1}$ has not changed, then~$\UU^{k+1}$ should draw its new measure from the old measure it removed the last time.  In this way~$\UU^{k+1}$ avoids putting excess computations on the measure of~$\UU^k$: besides the measure in~$\UU^{k+1}$, the measure in~$\UU^k$ which has a computation for~$A_s\uhr{k+1}$ because of~$\VV_n$ will be of size at most~$\leb(\VV_n)$.

To this end, we have~$\UU^k$ keep a \emph{bin} for every test component~$\VV_n$.  When measure is removed from~$\UU^{k+1}$ because it is being covered by~$\VV_n$, we put that measure into the bin for~$\VV_n$.  If~$\UU^k$ is called upon to furnish replacement measure because some measure is being covered by~$\VV_n$, then it must be that~$\VV_n$ has moved.  So~$\UU^k$ first looks for newly uncovered measure in the~$\VV_n$-bin to use as replacement measure.  In this way, it reuses measure as much as possible.

This lets us keep the size of the measure being risked small.  But suppose that the part of~$\UU^{k+1}$ which~$\VV_n$ covered was actually part of~$\UU^{\ell}$ for~$\ell$ much larger than~$k$.  Thus this measure has computations for~$A\uhr{\ell}$ on it.  Then~$A\uhr{\ell}$ can change~$\ell$ many times.  Even though we keep the measure being risked small, and so only a small amount of measure goes bad every time this changes, if there are enough changes this can add up to a large amount of measure going bad in total.  So we must work to make sure that the number of times measure can go bad is also kept small.

This, finally, is where we use the fact that~$A$ is strongly jump-traceable.  Every computation is some actor's responsibility.  Before we can put measure into a bin, that bin must take responsibility for the computations on that measure.  Before we move measure from a bin to a~$\UU^k$, that~$\UU^k$ must take responsibility for the computations on that measure.  Here taking responsibility means having tested the initial segment of~$A$ on some \emph{box} -- a part of a trace for an $A$-partial computable function.  Then the number of times measure can go bad is bounded by the \emph{size} of the box, namely, the chosen bound on the size of the trace component. 

\subsection{Outline of the construction}

The construction is performed on a tree of strategies. Firstly, nodes of length~$e$ on the tree measure whether~$g^e$ is total or not. Hence every node will have an outcome~$\fin$ (the~$\Sigma_2$ outcome), which believes that~$g^e$ is not total. It will have infinitely many outcomes which believe that~$g^e$ is total; we will go into more detail on these a little later.

Nodes on level $e$, together with their immediate children, are responsible for enumerating axioms into~$\Gamma$ which map clopen sets to strings of length $e$ (for a technical reason, this will only be for $e > 1$). The main tension in the construction is between the wish to define $\Gamma$ on large subsets of Cantor space, so that the Demuth tests don't cover the domain of $\Gamma$; and the need to keep the measure of $\EE$ small. That measure can increase if $\Gamma$ is defined on a big clopen subset of $2^\w$ and then $A$ changes. To minimise the ramifications of $A$-changes, before axioms mapping some clopen set to some $A_s\uhr{n}$ are enumerated into $\Gamma$, we \emph{test} the correctness of $A_s\uhr{n}$ using the strong jump-traceability of $A$. 

A node $\s$ will define an order function $h_\s$, which depends on $g^{|\s|}$. The immediate children of $\s$ will together define a p.c.\ functional $\Psi_\s$. We let $\seq{\bar{T}^{\s\conc d}}_{d<\w}$ be an enumeration of all $h_\s$-traces. The children of $\s$ which believe that $g^{|\s|}$ is total are $\s\conc d$ for $d<\w$; the child $\tau=\s\conc d$ guesses that $\bar{T}^{\tau} = \seq{T^{\tau}_z}_{z<\w}$ traces $\Psi_\s^A$. 

To test whether $\alpha = A_s\uhr{n}$ is an initial segment of $A$ using a prescribed input $z$, a child $\tau=\s\conc d$, accessible at stage $s$, sets $\Psi^\alpha_\s (z)  = \alpha$. The child then waits until a later stage $t$ at which we see that either $\alpha\nsubset A_t$ (the test failed), or that $\alpha\in T^{\tau}_z$ (the test succeeded).  If we run a test on a finite collection of inputs, we wait for the test on each input to return; of course success is declared in the case that $\alpha\subset A_t$, in which case on every input the test succeeded. There will also be a third possibility: the construction may cancel the test, in which case we simply stop considering those inputs.

At every stage $s$, every node $\s$ is equipped with a clopen set $\UU^\s_s$; these will eventually stabilise to a final value $\UU^\s$. The set $\UU^\s$ is where $\s$ defines $\Gamma$-computations. These sets form a tree of clopen sets: if $\s\subseteq \tau$ then $\UU_s^\tau\subseteq \UU^\s_s$, but if $\s$ and $\tau$ are incomparable, then $\UU^\s_s$ and $\UU^\tau_s$ are disjoint. 

Let $F$ be the set of nonempty nodes that do not end with $\fin$, that is, nodes of the form $\s\conc d$. For $\tau = \s\conc d$, we let $\seq{\VV^\tau_n} = \seq{\VV^{|\s|}_n}$, and $g^\tau = g^{|\s|}$.

For every $\tau\in F$, if ever accessible, we will choose some $n_\tau<\w$. This will be the point from which $\tau$ and its descendants have to avoid the test $\seq{\VV^\tau_n}$. For every node $\tau$, we let 
\[ F^\tau = \left\{ \s\in F\,:\, \s\subsetneq \tau \right\}; \]
we let 
\[
\WW^\tau_s = \UU^\tau_s \cap  \left(\bigcup V_{n,s}^\s \,\,\,\Cyl{\s \in F^\tau \andd n_\s \leq n \leq s}\right).
\]
This is a clopen subset of $\UU^\tau[s]$. Its ``limit'' is
\[
\WW^\tau = \UU^\tau \cap \left(  \bigcup \VV^\s_n \,\,\,\Cyl{\s \in F^\tau \andd n_\s \leq n}\right),
\]
which is open in $\UU^\tau$, but not effectively so. If $\tau$ lies on the true path, we need to ensure that $X\notin \WW^\tau$, and of course to make sure that $X\notin \EE$, so at the end, we let $X$ be the unique element in the intersection of the sets $\UU^\tau\setminus (\WW^\tau\cup \EE)$ where $\tau$ ranges along the true path. 

Of course, to do this, we need to ensure that for no $\tau$ do we get $\UU^\tau \subseteq \WW^\tau\cup \EE$. We do this by ensuring that the measure of $\left(\UU^\tau \cap (\WW^\tau\cup \EE)\right)[s]$ is smaller than the measure of $\UU^\tau_s$. We will set a rational number $\delta_\tau$, and if we see that 
\[ \leb\left(\UU^\tau\cap (\WW^\tau\cup \EE)\right)[s] \geq \delta_\tau,\] 
then all of $\UU^\tau\cap (\WW^\tau\cup \EE)[s]$ will be extracted from $\UU^\tau$ by the parent of $\tau$, and replacement measure of the same size will be given to $\tau$ by the parent. We set the measure of $\UU^\tau$ to be $4\delta_\tau$; so the aim is to ensure that the fraction of $\UU^\tau$ covered by $\WW^\tau\cup \EE$ is at most a quarter.

A parent $\tau$ providing its child with replacement measure has to be careful to recycle used measure; otherwise our efforts to limit the size of $\EE$ will fail. To do so, for every $\s\in F^\tau$, for every $n\ge n_\s$, $\tau$ keeps a \emph{bin} $\BB_s^\tau(n,\s)$. This consists of measure in~$\UU^\tau$ taken from children of $\tau$ on account of being covered by $\VV^\s_n$. Thus~$\BB_s^\tau(n,\s)$ is disjoint from $\UU_s^\rho$ for all immediate children~$\rho$ of~$\tau$.  When replacing more such measure, $\tau$ will first use this reserve of measure. We will thus ensure that even though $\VV^\s_n$ moves around a lot, the total measure in $\BB_s^\tau(n,\s)$ will never surpass~$2^{-2n}$. 

The point is that for the purposes of controlling the size of $\EE$, we need to \emph{charge} any clopen set $\CC$ in the domain of $\Gamma$ to some ``account'', namely sets of boxes (inputs) on which we run tests to verify that $\Gamma(\CC)$ is indeed an initial segment of~$A$. The sizes of the boxes (the bound on the possible size of the traces) limit the amount of drawing on the accounts, and so the amount of measure that can ``go bad'', i.e., into $\EE$. There are two kinds of accounts:
\begin{itemize}
	\item Measure in a bin $\BB_s^\tau(n,\s)$ is charged to $\s$'s account for dealing with $\VV^\s_n$. The limit we ensure on the amount of measure in this bin ensures that this account is not ``overdrawn''.  Even though~$\tau$ keeps this bin, since~$\sigma$ is the strategy which believes~$g^\sigma$ to be total,~$\sigma$ is responsible for providing the boxes which are used to test measure in~$B^\tau(n,\s)$.
	\item Measure distributed directly to some $\UU^\tau$ is charged to $\UU^\s$, where $\s$ is the longest element of $F^\tau$. This is why we assumed that~$g^0$ is total:  we can thus assume that~$F^\tau$ is non-empty for every visited~$\tau$.
\end{itemize}

This is the core of the construction.  All that remains is setting up the numbers such that the arithmetic works out.

\subsection{Defining~$\delta_\s$ and~$n_\s$}

We distribute waste targets among nodes. Let $\seq{\s_n}$ be an effective enumeration of all nodes. Let $\epsilon_{\sigma_n} = 2^{-(n+5)}$. We will require that the total amount of bad measure charged by a node $\sigma$ is at most $2\epsilon_\sigma$, one $\epsilon_\sigma$ for each kind of ``account''. The definition of $\epsilon_\sigma$ is made so that 
\[ \sum_{\s} 2\epsilon_\s \le \frac14 .\]

We describe how to define the numbers $n_\tau$ and $\delta_\tau$. Along with these, we define auxiliary numbers $\delta^\s_\tau$ for $\s\in F$ and $\tau\supset \s$. The idea is the following: a node $\s\in F$ introduces the tail of a test $\seq{\VV^\s_n}_{n\ge n_\s}$, which, as described above, the extensions of $\s$ have to avoid. The associated replacements between parents and children nodes due to this tail may add measure to $\EE$. To ensure that the total amount that can go bad due to this tail-of-test is small, $\s$ instructs $\tau\supset \s$ to keep its $\delta_\tau$ below $\delta^\s_\tau$. 

In defining $n_\tau$ and $\delta_\tau$, we need to ensure that:
\begin{itemize}
	\item $n_{\s\conc d}\ge d$ (as $\s\conc d$ will have access to $d$-boxes, and will require access to $n_{\s\conc d}$-boxes).
	\item If $\delta_{\s\conc d}=2^{-2k}$ then measure enumerated into $\EE$ because of $\Gamma$-computations under the direct responsibility of $\s\conc d$ will be bounded by $2^{-k}$ (the increase is due to \emph{repeated} losses by the same actor, due to the fact that this actor does not have access to 1-boxes). The total over all such $d$ has to be bounded by $\epsilon_\sigma$.
	\item Let $\s\in F$ and let $\tau\supseteq \s$. The bin $\BB^\tau(n,\s)$ may have measure at most $\min \{\delta_\tau,2^{-2n}\}$ (recalling that $\leb(\VV^\s_n)\le 2^{-2n}$). If that size is at most $2^{-2k}$, we will again ensure that the total measure going into $\EE$ due to that bin is at most $2^{-k}$. We will need to devise the numbers $\delta^\s_\tau$ so that the total such amount, over all bins corresponding to $\seq{\VV^\s_n}_{n\ge n_\s}$, is at most $\epsilon_\s$. 
	\item The compensation from $\s$'s parent ensures that $\s$ always has at least $3\delta_\s$ much measure disjoint from $\EE\cup \WW^\s$. This has to be distributed among $\s$'s children, with some measure left for $\s$ for compensating children when $\s$ comes to their aid as their share gets covered by bad measure. Further, measure covered by the test introduced by $\s$ ($\seq{\VV_n^\s}_{n\ge n_\s}$) is unfit for use for $\s$'s children, as they have to avoid this test. Hence the total sum of $\leb(\UU^\tau)$ for the children $\tau$ of $\s$ has to be bounded by say $\delta_\s$; and the total measure covered by $\s$'s test also has to be smaller than say $\delta_\s$. 
\end{itemize}

We start by letting $\delta_{\seq{}} = 1/4$ (here $\seq{}$ denotes the empty string, that is, the strategy at the root of the tree).  We then proceed recursively. 
\begin{enumerate}
	\item Let $\tau$ be any node and suppose that $\delta_\tau$ is defined. Let $m_\fin$ be the least number greater than 5 such that $2^{-m_\fin} \le \epsilon_{\tau\conc\fin}$. Let	
\[
\delta_{\tau\conc \fin} = \min\big\{ \delta_\tau \cdot 2^{-(2m_\fin)}, \delta^\s_{\tau\conc \fin}\,:\, \s\in F^{\tau\conc \fin}\big\}
\]
and for $d<\w$, let $m_d$ be the least number greater than 5 such that $2^{-m_d} \le \epsilon_{\tau\conc d}$ and let
\[
\delta_{\tau\conc d} = \min\left\{ \delta_\tau \cdot 2^{-(2m_d+d)}, \delta^\s_{\tau\conc d}\,:\, \s\in F^{\tau\conc d}\right\} .
\]
\item Let $\s\conc d\in F$, and suppose that $\delta_{\s\conc d}$ is already defined. Let $n_{\s\conc d}$ be the least even number $n$ greater than $\max\{d,10\}$ such that:
	\begin{enumerate}
		\item $6\cdot 2^{-n}\le  \epsilon_{\s\conc d}$; and
		\item $2\cdot 2^{-n} \le \delta_{\s\conc d}$. 
	\end{enumerate} 
	Now enumerate all the proper extensions of $\s\conc d$ as $\seq{\tau_k}_{k\ge 1}$; let $\delta^{\s\conc d}_{\tau_k} = 2^{-2(n_{\s\conc d}+k)}$. 
\end{enumerate}
Observe that the~$\delta_\s$ are all integer powers of~$2$.

\subsection{Bounding the changes to~$\UU^\s$}\label{sec:defining p}

We construct a bound~$p(\s)$ on the number of times~$\UU^\s$ changes.  We define $p(\seq{}) = 1$. 

Below, and for the rest of the paper, we let~$\s^-$ denote the immediate predecessor of a node~$\s$; so $\s^- = \s\uhr{|\s| - 1}$. We shall be careful to only write~$\sigma^-$ when $\sigma \neq \seq{}$.

Whenever~$\UU^{\s^-}$ changes, $\UU^\s$ is made empty.  Also, new measure is given if $\UU^\s_s = \emptyset$, which it will only be at the beginning of the construction or when~$\UU^{\s^-}$ changes.  Thus changes of this sort can occur at most $2\cdot p(\s^-)$ many times.

New measure is also given in case the measure of $\left(\UU^\s\cap (\WW^\s\cup \EE)\right)[s]$ exceeds~$\delta_\s$. For this to happen, we must have either $\leb(\WW^\s_s)\ge \delta_\s/2$ or $\leb(\UU^\s \cap \EE)[s] \ge \delta_\s/2$. Since $\EE$ is effectively open, the latter can happen at most $2/ \delta_\s$ many times: each time it happens, a subset of $\EE_s$ of size $\delta_\s/2$ is removed from $\UU^\s_s$, and it is never returned to $\UU^\s$, as new measure supplied to $\s$ after stage $s$ is disjoint from $\EE_s$.

Now consider $\WW^\s_s$. Since the tails of the tests $\seq{\VV^\rho_n}_{n\ge n_\rho}$ which make up $\WW^\s$ shrink quickly, we can easily compute a number $m$ such that if $\leb(\WW^\s_s)\ge \delta_\s/2$, then 
\begin{equation*}\tag{$\dagger$}\label{size-change-ineq}
\leb \left( \bigcup \VV^\rho_{n,s} \,\,\,\Cyl{\rho \in F^\s \andd n_\rho \leq n \leq m} \right) \ge \delta_\s/4;
\end{equation*}
for example $m > 8|F^\s| \cdot (-\log_2 \delta_\s)$ should do.  Measure can only leave a bin~$\BB^{\s^-}(n,\rho)$ when either~$\VV_n^\rho$ changes or~$\UU^{\s^-}$ changes.  Thus, between stages at which~$\UU^{\s^-}$ or any of the $\VV^\rho_{n}$ (for $\rho\in F^\s$ and $n\in [n_\rho,m]$) change, the situation in (\ref{size-change-ineq}) can happen at most $|F^\s|\cdot m\cdot 4/\delta_\s$ many times. The total number of times any of these sets change is of course at most
\[ p(\s^-) + \sum g^\rho(n) \,\,\,\Cyl{\rho \in F^\s \andd n_\rho \leq n \leq m}.\]
Hence we can let the bound 
\[ p(\s) = 2\cdot p(\s^-) + 2/\delta_\s + |F^\s|\cdot m\cdot 4/\delta_\s \left( p(\s^-) + \sum g^\rho(n) \,\,\,\Cyl{\rho \in F^\s \andd n_\rho \leq n \leq m}\right).\]

Of course,~$p$ will be partial computable, since~$p(\s)$ will only exist if all of the~$g^\rho(n)$ in the above expression exist.  However, if any of the~$g^\rho$ for~$\rho \in F^\s$ are partial, we know that~$\s$ is not on the true path.  In fact, for such~$\rho$ we would know that $\rho \subset \s$ is not on the true path.  Thus any~$\s$ can wait until~$p(\s)$ is defined before taking any action.

\subsection{Definition of~$h_\s$}

We now describe the order functions defined by nodes. Let $\s$ be a node. For all $d<\w$ and all $n\ge n_{\s\conc d}$ we will define a number $m_{\s\conc d}(n)$ -- this is the number of $n$-boxes required of $\s$ by $\s\conc d$. We then define the order function $h_\s$ by ensuring that for all $n$, 
\[ \# \left\{ z\,:\, h_\s(z) = n \right\} = \sum m_{\s\conc d}(n) \,\,\,\Cyl{n_{\s\conc d}\le n} .\]
Note that since $n_{\s\conc d}\ge d$, the sum on the right is finite, so the function $h_\s$ will be unbounded. 

Recall that there are two streams pouring measure into $\EE$: measure from bins, and measure allocated as $\UU^\tau$. For either stream, we need to allocate boxes for tests ensuring that the stream is not too voluminous.

We start with bins.  Fix $d<\w$ and recall that $g^{\s\conc d} = g^{|\s|}$.  Let $\rho\supseteq \s\conc d$ and $n\ge n_{\s\conc d}$. The total measure in the bin $\BB^\rho(n,\s\conc d)$ is, as mentioned above, $\min \{ \delta_\rho,2^{-2n}\}$. This is a number of the form $2^{-k}$. Since the bin is tied to the test component $\VV^{\s\conc d}_n$ which can move around at most $g^{\s\conc d}(n)$ many times, and since it will be convenient to discard our existing tests whenever~$\UU^\rho$ receives new measure, for every $k'\ge k$ we will require 
\[ \left(1+p(\rho)\cdot 2^{k'+1}\cdot g^{\s\conc d}(n)\right)^{k'+1} \] many $k'$-boxes from $\s$.

Next, we deal with measure allocated as $\UU^\tau$. The intended measure $4\delta_{\tau}$ of $\leb(\UU^{\tau})$ is again a number of the form $2^{-k}$ for some $k$. The child $\tau$ requires access to $k$-boxes when it is first set up, and when it is replenished with measure by $\tau^-$.  Thus we will require  
\[ \left(1+ p(\tau^-) + p(\tau)\right)^{k+1} \]
extra $k$-boxes for $\tau$.

We do not in general request these boxes as part of~$m_\tau(k)$ (only in part because~$\tau$ may end in~$\fin$).
Instead, let~$\s\conc d \in F^\tau$ be largest.
These $k$-boxes are incorporated into~$m_{\s\conc d}(k)$.  We observe that $F^\tau \subseteq F^{\s\conc d} \cup \{\s\conc d\}$, and thus~$p(\tau)$ depends only on~$\delta_\tau$ and~$g^\rho$ for~$\rho \in F^\s \cup \{\s\conc d\}$.  Also, since $\delta_\tau$ approaches 0 rapidly, for a given~$\s\conc d$ and~$k$ there are only finitely many extensions~$\tau$ which will seek to incorporate a request into~$m_{\s\conc d}(k)$, and further we can uniformly compute (a canonical index for) the collection of these extensions.

The astute reader may observe that there is a problem with the above paragraph if~$\tau^- = \fin^n$ for some~$n$, or if $|\tau| \leq 1$.  However, as we have assumed that~$g^0$ is total, we will ignore $\tau$ of the first sort on the tree.  For~$\tau$ of the second sort, since there is no~$V^\seq{}$, and~$\tau$ does not enumerate computations, there is no manner in which measure~$\UU^\tau$ is responsible for might enter $\EE$.  Thus~$\UU^\tau$ has no need of boxes.

This completes the definition of the numbers~$m_{\s\conc d}(n)$, and so of~$h_\s$.  We observe that~$m_{\s\conc d}$ is a partial computable function, and is total if~$g^\rho$ is total for all $\rho \in F^\s \cup \{\s\conc d\}$.  Of course, if one or more of these functions is not total, then none of the infinitary children $\s\conc d$ of $\s$ lies on the true path. Hence we define $h_\s$ as the construction proceeds, extending it during $\s$-expansionary stages. The above calculation gives us, for every $k$, a length $\ell_\s(k)$ such that given $g^\rho\rest{\ell_\s(k)}$ for all $\rho\in F^\s\cup \{\s\conc d\}$, we know how many $k$-boxes are required by all of $\s$'s children together, and so only once we see convergence of these functions up to $\ell_\s(k)$ do we define $h_\s$ to equal $k$ on the required interval. Before these $k$-boxes are ``set up'', we do not allow any child of $\s$ that requires them to take any action. 

\subsection{Organizing the boxes}

For $|\s| > 1$, let~$k$ be such that $2^{-k} = 4\delta_\s$, and let~$\pi\conc d \in F^\s$ be largest, so that~$\s$'s extra~$(1+ p(\s^-) + p(\s))^{k+1}$ many~$k$-boxes were incorporated into~$m_{\pi\conc d}(k)$.  We specify an interval $I_\s \subseteq h_\pi^{-1}(k)$ of size $(1+ p(\s^-)+p(\s))^{k+1}$.  We think of~$I_\s$ as the discrete hyper-cube $[0, p(\s^-) +p(\s)]^{k+1}$.  $I_\s$ will be used for testing the measure which enters~$\UU^\s$.

Also for $|\s| > 1$, for $\rho \in F^\s$, $n \ge n_\rho$ and $k$ with $2^{-k} \le \min\{\delta_\s, 2^{-2n}\}$, we specify an interval $J_\s(\rho, n, k) \subseteq h_{\rho^-}^{-1}(k)$ of size $(1+p(\rho)\cdot2^{k+1}\cdot g^\rho(n))^{k+1}$.  These are the boxes set aside for $\BB^\s(n,\rho)$.  We think of $J_\s(\rho,n,k)$ as the discrete hyper-cube $[0, p(\s)\cdot2^{k+1}\cdot g^\rho(n)]^{k+1}$.  We use~$J_\s(\rho,n,k)$ for testing the measure which enters~$\BB^\s(n,\rho)$.

These intervals are all chosen to be disjoint. They can be easily computed once the appropriate numbers $m_{\s\conc d}(n)$ are known.

\subsection{Organizing the tests}

As discussed before, we have intervals which we think of as discrete hyper-cubes on which we perform tests.  We perform all tests on an affine hyper-plane of said hyper-cube.  Whenever a test succeeds, we move on to the next hyper-plane.  However, if a test succeeds but~$A$ later changes such that the tested~$\alpha$ is not an initial segment of~$A$, we restrict our attention to the hyper-plane used in that test.  That hyper-plane becomes our new hyper-cube, and all future testing is done within it.

In order to properly track which boxes we perform tests on, we will keep several numbers.

Let~$H$ be a hyper-cube (some~$I_\s$ or~$J_\s(\rho,n,k)$).  For tests on~$H$, we keep a number $b_s(H)$, which is the number of times we have restricted to a hyper-plane by stage~$s$.  We shall arrange that $b_s(H) \leq k$, where~$k$ is the size of the boxes which make up~$H$ ($k+1$ is the dimension of~$H$).  We also keep numbers~$c_s(H,i)$ for every $i \leq k$, which will indicate which hyper-plane we restricted to each time, and where we are currently testing.  In practice,~$c_s(H,i)$ will count tests which succeeded or were cancelled.

We begin by defining~$b_0(H) = 0$ and $c_0(H,i) = 0$ for every $i \leq k$.

At stage~$s+1$, if there are stages $t_0 < t_1 < s$ such that $b_{t_0}(H) = b_s(H)$ and a test on~$H$ of some~$\alpha$ was begun at stage~$t_0$ and succeeded at stage~$t_1$, but $\alpha \not \subset A_{s+1}$, then~$b_{s+1}(H) = b_s(H) + 1$.  We choose the least such~$t_0$ and let $c_{s+1}(H,b_s(H)) = c_{t_0}(H,b_s(H))$.

Otherwise, we define $b_{s+1}(H) = b_s(H)$.  If there is a stage $t < s$ such that $b_t(H) = b_{s+1}(H)$ and $c_t(H,b_{s+1}(H)) = c_s(H,b_{s+1}(H))$ and a test on~$H$ was begun at stage~$t$ and succeeded or was cancelled at stage~$s$, we define~$c_{s+1}(H,b_{s+1}(H)) = c_s(H,b_{s+1}(H)) + 1$.  Otherwise, we define $c_{s+1}(H,b_{s+1}(H)) = c_s(H,b_{s+1}(H))$.

For $i < b_{s+1}(H)$, we always define $c_{s+1}(H,i) = c_s(H,i)$.

When we wish to perform a test on~$H$ at stage~$s+1$, we shall test on the subspace
\[
\{\vec{x} \in H \mid \forall i \le b_{s+1}(H) [x_i = c_{s+1}(H,i)]\}.
\]

If we wish to perform a test, but the interval~$H$ has not yet been defined (because the appropriate order~$h_\s$ has not yet been sufficiently extended), we wait until a stage at which~$H$ has been defined, and then immediately perform the test.

\subsection{Moving measure into bins}

Whenever measure enters~$\BB^\s(n,\rho)$, the bin must claim responsibility for that measure with some certainty~$k'$ (this is the size of the box on which this measure was tested).  We maintain three properties:
\begin{itemize}
\item The amount of measure in~$\BB^\s_s(n,\rho)$ never exceeds $\min\{\delta_\s, 2^{-2n}\}$.
\item All measure in~$\BB^\s_s(n,\rho)$ is claimed with some certainty~$k$.
\item The amount of measure claimed with certainty~$k$ is no more than~$2^{-k+1}$.
\end{itemize}
We describe how this is maintained.

Suppose, at stage~$s$, strategy~$\s$ has some clopen set~$\XX$ that it wishes to move into~$\BB^\s_{s+1}(n,\rho)$ (because $\XX \subseteq (\UU^\s \cap \VV_n^\rho)[s]$).  As discussed before, we strive to reuse measure in~$\BB^\s(n,\rho)$ whenever possible.  Thus, when~$\XX$ enters~$\BB^\s_{s+1}(n,\rho)$, any measure in~$\BB^\s_s(n,\rho)$ not covered by~$\VV_{n,s}^\rho$ is used first to replace the measure in~$\XX$.  We choose clopen $\YY \subseteq \left(\BB^\s(n,\rho) - \VV_{n}^\rho\right)[s]$ with
\[
\leb(\YY) = \min\{ \leb(\BB^\s(n,\rho) - \VV_{n}^\rho)[s], \leb(\XX)\}.
\]
If we were to move measure as desired, $\BB^\s_{s+1}(n,\rho) = \BB^\s_s(n,\rho) \cup \XX - \YY$.  However, we cannot move this measure until we have successfully tested it.

We let~$k_0$ be least such that $2^{-k_0} \le \leb(\BB^\s_s(n,\rho) \cup \XX - \YY)$.  We would like~$\BB^\s(n,\rho)$ to take responsibility for~$\BB^\s_{s+1}(n,\rho)$ with certainty~$k_0$, but in order to avoid running out of~$k_0$ boxes, we must consider the possibility that~$\BB^\s(n,\rho)$ is already taking responsibility for some subset of~$B^\s_s(\rho,n) - \YY$ with certainty~$k_0$.  Let~$\ZZ(k_0)$ be the clopen subset of~$\BB^\s_s(n,\rho) - \YY$ for which~$\BB^\s(n,\rho)$ is already taking responsibility with certainty~$k_0$.  We let~$k_1$ be least such that
\[
2^{-k_1} \le \leb(\BB^\s_s(n,\rho)\cup \XX - \YY - \ZZ(k_0)),
\]
and let~$\ZZ(k_1)$ be the clopen subset of $\BB^\s_s(n,\rho) - \YY - \ZZ(k_0)$ for which~$\BB^\s(n,\rho)$ is taking responsibility with certainty~$k_1$.  In this way we create a sequence $k_0\leq k_1\leq k_2\leq \dots$, which we continue until finding an~$m$ such that $k_m = k_{m+1}$.  Note that this will necessarily happen, as at any finite stage,~$\BB^\s(n,\rho)$ will have only used finitely many~$k$.

Having found~$k_m$, we let~$r$ be largest such that~$\Gamma(\CC) = A_s\uhr{r}$ for some clopen $\CC \subseteq \UU^\s$, or $r = |\s| + 1$, whichever is larger.  We test $\alpha = A_s\uhr{r}$ on~$J_\s(\rho,n,k_m)$.  If, before this test returns,~$\UU^\s$ changes, we cancel the test and end this attempt to move~$\XX$ into~$\BB^\s(n,\rho)$.  Otherwise, because of how we shall define accessible strategies, at the next stage~$t$ at which~$\s$ is accessible, we shall be guaranteed that the test has returned.

If the test failed, our attempt to move~$\XX$ into~$\BB^\s(n,\rho)$ has failed, and we do nothing more.

If the test succeeded, we move~$\XX$ into~$\BB^\s(n,\rho)$.  Our construction will ensure that $\BB^\s_s(n,\rho) = \BB^\s_t(n,\rho)$.  We define
\[
\BB^\s_{t+1}(n,\rho) = \bigl(\XX \cup \BB^\s_t(n,\rho)\bigr) - \YY.
\]
$\BB^\s(n,\rho)$ takes responsibility with certainty~$k_m$ for
\[
\BB^\s_{t+1}(n,\rho) - \ZZ(k_0) - \ZZ(k_1) - \dots - \ZZ(k_{m-1}).
\]

\subsection{Construction}

For all $s$, we let $\UU^{\seq{}}_s = 2^\w$.  For all $\s, \rho, n$ with $|\s| > 0$, we let $\UU_0^\s = \emptyset$ and $\BB^\s_0(n,\rho) = \emptyset$.

At stage $s$, we describe which nodes~$\s$ are accessible at stage $s$, and what actions they take. These will be:
\begin{enumerate}
	\item Extending the definition of $h_\s$;
	\item Defining $\UU^\tau_{s+1}$ for children $\tau$ of $\s$;
	\item Defining the bins $\BB^\s_{s+1}(n,\rho)$ for $\rho\in F^\s$; and
	\item Enumerating axioms into $\Gamma$;
	\item Choosing a child~$\tau$ to attend to next.
\end{enumerate}

The root $\seq{}$ is accessible at every stage. We will later show (Lemma \ref{lem:total bad measure}) that $\leb(\EE_{s})\le 1/4$. Based on this, we proceed with the stage. Let~$\s$ be a node which is accessible at stage $s$. If $|\s|=s$ we halt the stage.  If~$\s = \fin$, we halt the stage (as by assumption~$g^0$ is total, and thus~$\fin$ cannot be the true outcome of~$\seq{}$).

Otherwise, we break the action of~$\s$ into three substages.

\subsubsection{First Substage}

First we extend~$h_\s$, as already discussed.



\subsubsection{Second Substage}

We let $t < s$ be the last stage at which~$\s$ was accessible (with $t = 0$ if there was no such previous stage).  At stage~$t$,~$\s$ may have begun tests because it wished to change~$\UU^\tau$ for some child~$\tau$.  If there were no tests begun or they were cancelled before stage~$s$, we proceed immediately to the third substage.

If there were tests begun which were not cancelled, then those tests have now returned.  If they failed, we proceed immediately to the third substage.

Otherwise, we change~$\UU^\tau_{s+1}$ and the various bins~$\BB^\s_{s+1}(n,\rho)$ as~$\s$ wished to at stage~$t$.  For every $\pi \supseteq \tau$ and every~$\rho$ and~$n$, we define~$\BB^\pi_{s+1}(n,\rho) = \emptyset$ and cancel any tests begun by~$\pi$.  For every~$\pi \supsetneq \tau$, we define $\UU^\pi_{s+1} = \emptyset$.  If $|\tau| > 1$, we enumerate the axiom $\Gamma(\UU^\tau_{s+1}) = A_s\uhr{|\tau|}$.  We then end stage~$s$.

\subsubsection{Third Substage}

First we choose a child~$\tau$ to attend to by considering $\dom g^{|\s|}$.  Let $t < s$ be the last stage at which~$\s$ was accessible (with $t = 0$ if there was no such previous stage).  If $\dom g_t^{|\s|} = \dom g_s^{|\s|}$, we let $\tau = \s\conc\fin$.  Otherwise, we choose the least~$d$ such that no test on~$\Psi_\sigma$ begun by~$\s\conc d$ or an extension of~$\s\conc d$ is waiting for a return in~$\bar{T}^{\s\conc d}$, and no test on~$\Psi_{\sigma\conc d}$ is waiting to begin (because the appropriate interval has not yet been defined).  We observe that $d = s$ is always such a~$d$, since by induction~$\s\conc s$ has never yet been accessible.  We let $\tau = \s\conc d$.

We consider $\leb(\UU^{\tau} - (\WW^{\tau} \cup \EE))[s]$.  There are two cases.

\

{\em Case 1:}

If $\leb (\UU^{\tau} - (\WW^{\tau} \cup \EE))[s] \leq 3\delta_\tau$ (this includes the possibility that $\UU^\tau_s = \emptyset$), we wish to extract $\UU^{\tau} \cap (\WW^{\tau} \cup \EE))[s]$, 
moving~$\WW^\tau_s$ into the appropriate bins and moving new measure into~$\UU^\tau_{s+1}$.  All measure in~$\WW^\tau_s$ must be moved into a bin, but potentially measure could properly be placed in more than one bin.  Thus we partition~$\WW^\tau_s$ as $\WW^\tau_s = \bigsqcup_{\rho,n} \XX_{n,s}^\rho$, where
\[
\XX_{n,s}^\rho = \left(\WW^\tau \cap \VV_{n}^\rho - \bigcup \XX_{n'}^\pi \,\,\,\Cyl{\langle\pi,n'\rangle < \langle \rho, n\rangle}\right)[s]
\] is clopen, $\rho \in F^\tau$, and $n_\rho \leq n \leq s$.  The particulars of the ordering~$<$ are unimportant.

We wish to move~$\XX^\rho_{n,s}$ into~$\BB^\s(n,\rho)$ (as previously discussed), and we simultaneously wish to move new measure into~$\UU^\tau$ to replace that which was removed.  Let~$\YY^\rho_{n,s}$ be the clopen set~$\YY$ provided by~$\BB^\s(n,\rho)$.  These may not be enough to bring the measure of~$\UU^\tau$ up to~$4\delta_\tau$, so we choose clopen
\[
\YY^\tau_s \subseteq \UU^\s_{s+1} - \left( \WW^\s \cup \bigcup_{n \ge n_\s} \VV_{n}^\s \cup \bigcup_d \UU^{\s\conc d} \cup \UU^{\s\conc \fin}\cup \EE\right)[s]
\]
with
\[
\leb\left( \UU^\tau \cup \YY^\tau \cup \bigcup_{\rho, n} \YY_{n}^\rho - \WW^\tau - \EE\right)[s] = 4\delta_\tau.
\]
We wish to move~$\YY^\tau_s$ and the~$\YY^\rho_{n,s}$ into~$\UU^\tau$, but first~$\UU^\tau$ must be prepared to take responsibility for this measure.  Let~$r$ be largest such that~$\Gamma(\CC) = A_s\uhr{r}$ for some clopen $\CC \subseteq \UU^\s$, or $r = |\tau|$, whichever is larger.  We begin a test of $\alpha = A_s\uhr{r}$ on~$J_\tau$ (observe that this is the same~$\alpha$ we are testing to move~$\XX^\rho_{n,s}$ into~$\BB^\s(n,\rho)$).  We then end the stage.

\

{\em Case 2:}

If $\leb (\UU^{\tau} - (\WW^{\tau} \cup \EE))[s] > 3\delta_\tau$, we do not need to adjust~$\UU^\tau$.  We define $\UU^\tau_{s+1} = \UU^\tau_s$.  We let~$\tau$ be the next accessible node and continue the stage.

\subsubsection{End of Stage}

At the end of stage~$s$, for every~$\UU^\s_{s+1}$ which was not otherwise defined, we define $\UU^\s_{s+1} = \UU^\s_s$.  For every~$\BB^\s_{s+1}(n,\rho)$ which was not otherwise defined, we define $\BB^\s_{s+1}(n,\rho) = (\BB^\s(n,\rho) - \EE)[s]$.

\subsection{The True Path and $X$}

We define the true path through the construction inductively.  If~$g^e$ is partial, define~$f(e) = \fin$.  Otherwise, define~$f(e) = d$ where~$d$ is least such that~$(f\uhr{e})\conc d$ is infinitely often accessible (we show in Claim \ref{claim:tot outcome exists} that such a~$d$ exists).

Define~$X$ to be the unique element of~$\bigcap \left(\UU^{f\uhr{e}} - \WW^{f\uhr{e}} - \EE\right) \,\,\,\Cyl{e \in \w}$, where~$\UU^{\s} = \lim_s \UU^\s_s$.  We show in Claim \ref{claim:U limit exists} that~$\UU^\s$ exists and in Claim \ref{claim:intersection nonempty} that this intersection is non-empty.  That~$X$ is unique follows from the shrinking sizes of the~$\UU^{f\uhr{e}}$.

\subsection{Verification}

We perform the verification as a sequence of claims.

\begin{cclaim}
For~$\s$ a strategy in the construction, if~$p(\s)$ is undefined, then~$\UU^\s_s = \emptyset$ for every~$s$.
\end{cclaim}

\begin{proof}
By construction, before~$\s$ will receive measure,~$\s^-$ must have a test on~$I_\s$ succeed.  Since~$p(\s)$ is undefined,~$I_\s$ is never defined.  Thus any test on~$I_\s$ that~$\s^-$ wishes to perform will wait forever to begin, and hence will never return successfully.
\end{proof}

\begin{cclaim}\label{claim:U limit exists}
For~$\s$ a strategy in the construction, if~$p(\s)$ is defined, then~$p(\s)$ bounds the number of times~$\UU^\s$ changes.
\end{cclaim}

\begin{proof}
Immediate from the calculations in Section \ref{sec:defining p}.
\end{proof}

\begin{cclaim}
If~$H$ is some discrete hyper-cube used for testing, and~$s$ is some stage of the construction, every box in the subspace
\[
\{\vec{x} \in H \mid \forall i \le b_{s}(H) [x_i = c_{s}(H,i)]\}
\]
contains at least~$b_{s}(H)$ many incorrect values (values distinct from the value of~$\Psi_\s^A$) at stage~$s$.
\end{cclaim}

\begin{proof}
Immediate by induction on the value of~$b_s(H)$.
\end{proof}

\begin{cclaim}
For~$\s$ a node on the tree such that~$p(\s)$ and~$I_\s$ exist, and any~$i$, $c_s(I_\s,i) \leq p(\s^-) +p(\s)$.
\end{cclaim}

\begin{proof}
$c_s(I_\s,i)$ only increases when a test on~$I_\s$ succeeds or is cancelled.  A successful test indicates that~$\UU^\s$ changes, and so there are at most~$p(\s)$ many of those.  By construction, a test is only begun if~$\UU^{\s^-}$ is nonempty, and is only cancelled when some~$\UU^\pi$ changes, with $\pi \subset \s$.  But if $\pi \subset \s^-$, then $\UU^{\s^-}$ is set empty when~$\UU^\pi$ changes.  Thus every cancelled test corresponds to a change in~$\UU^{\s^-}$, and there are at most~$p(\s^-)$ many of those.
\end{proof}

\begin{cclaim}
For nodes~$\s$ and~$\rho$ and values~$n$ and~$k$ such that~$p(\s)$,~$g^\rho(n)$ and $J_\s(\rho,n,k)$ exist, and any~$i$, $c_s(I_\s,i) \leq p(\s)\cdot 2^{k+1}\cdot g^\rho(n)$.
\end{cclaim}

\begin{proof}
$c_s(J_\s(\rho,n,k),i)$ only increases when a test on~$J_\s(\rho,n,k)$ succeeds or is cancelled.  A successful test indicates that~$\BB^\s(n,\rho)$ takes responsibility for some measure entering~$\BB^\s(n,\rho)$ with certainty~$k$.  By construction, this amount of measure is at least~$2^{-k}$.  Thus, without changing~$\VV_n^\rho$ or emptying~$\BB^\s(n,\rho)$, this can occur at most~$2^{k}$ many times.  Changes to~$\VV_n^\rho$ happen at most~$g^\rho(n)$ many times, and~$\BB^\s(n,\rho)$ is only emptied when~$\UU^\s$ changes, which occurs at most~$p(\s)$ many times.

A cancelled test, meanwhile, means that some amount of measure (at least $2^{-k}$) attempted to enter~$\BB^\s(n,\rho)$ with certainty~$k$, but then~$\UU^\s$ changed.  This can occur at most once per change of~$\UU^\s$, and when it does, it indicates that one of the potentially~$2^{k}\cdot g^\rho(n)$ many successful tests did not occur for this version of~$\UU^\s$.

Thus $p(\s)\cdot 2^{k+1}\cdot g^\rho(n)$ is a suitable upper bound.
\end{proof}

\begin{cclaim}
The amount of measure in~$\BB^\s(n,\rho)$ never exceeds $\min\{\delta_\s,2^{-2n}\}$.
\end{cclaim}

\begin{proof}
Suppose not.  Let~$t$ be least such that~$\leb(\BB^\s_{t+1}(n,\rho)) > \min\{\delta_\s, 2^{-2n}\}$, and let $s < t$ be the stage at which~$\s$ realized it wished to move measure into~$\BB^\s(n,\rho)$ (the stage when~$\s$ first desired to perform the tests that later resulted in adding measure to~$\BB^\s_{t+1}(n,\rho)$). By construction, since $\leb(\BB^\s_{t+1}(n,\rho)) >  \leb(\BB^\s_t(n,\rho))$, $\leb(\XX_{n,s}^\rho) > \leb(\YY_{n,s}^\rho)$.  So we know $\BB^\s_{t+1}(n,\rho) \subseteq \VV_{n,s}^\rho$ (otherwise~$\YY_{n,s}^\rho$ would have been chosen larger).  But $\leb(\VV_{n,s}^\rho) \leq 2^{-2n}$, so $\leb(\VV_{n}^\rho \cap \UU^\s)[s] > \delta_\s$.

If $\s = \rho$, then by choice of~$n_\rho$, $\leb(\VV_{n,s}^\rho) < \delta_\s$, a contradiction.  Otherwise, when~$\s^-$ was accessible at stage~$s$, it would have seen $\leb(\UU^\s \cap (\WW^\s \cup \EE))[s] > \delta_\s$ and acted, preventing~$\s$ from being accessible at stage~$s$, which is a contradiction.
\end{proof}

\begin{cclaim}
All measure in~$\BB^\s_s(n,\rho)$ is claimed with some certainty~$k$.
\end{cclaim}

\begin{proof}
By construction, measure never enters~$\BB^\s(\rho,n)$ without being claimed.
\end{proof}

\begin{cclaim}
The amount of measure in~$\BB^\s_s(n,\rho)$ claimed with certainty~$k$ is no more than~$2^{-k+1}$.
\end{cclaim}

\begin{proof}
Suppose not.  Let $t < s$ be the stage at which the amount of measure claimed with certainty~$k$ increased beyond~$2^{-k+1}$.  Then by construction,~$k_m$ would have been chosen to be $\leq k-1$ instead of~$k$.
\end{proof}

\begin{cclaim}
For $\beta \in 2^\w$, if $\beta \in \UU^\s_s$ with~$|\s| \le 1$, but $\beta \not \in \UU^\tau_s$ for any~$\tau$ extending~$\s$, and $\beta \not \in \EE_s$, and $\beta \not \in \BB^\s_s(n,\rho)$ for any~$n$ and~$\rho$, then there is no neighborhood~$\CC$ of~$\beta$ such that~$\Gamma_s(\CC)$ is defined.
\end{cclaim}

\begin{proof}
By construction, in order for~$\Gamma_s$ to be defined on a neighborhood of~$\beta$, $\beta$ must have been an element of~$\UU^\tau_t$ for some~$|\tau| = 2$ and some $t < s$.  But the only way~$\beta$ can leave~$\UU^\tau$ before stage~$s$ is to enter~$\EE$ or some $\BB^\s(\rho,n)$.  And the only way~$\beta$ can leave~$\BB^\s(n,\rho)$ is to enter some~$\UU^{\tau'}$ with $|\tau'| = 2$.
\end{proof}

\begin{cclaim}\label{lem:total bad measure}
$\leb(\EE) \le 1/4$.
\end{cclaim}

\begin{proof}
Suppose a real~$\beta$ enters~$\EE$ at some stage~$s$.  Then it must be that~$\Gamma^\beta = \alpha$ for some~$\alpha$, and $\alpha \prec A_{s-1}$, but $\alpha \not \prec A_s$.  Let~$\s$ be largest such that~$\beta \in \UU^\s_s$ (since $\UU^\seq{}_s = 2^\w$, there is always such a~$\s$).  By the previous claim, $|\sigma| > 1$.

There are two cases.  The first possibility is that $\beta \in \BB^\s_s(n,\rho)$ for some~$\rho$ and~$n$.  In this case,~$\BB^\s(n,\rho)$ took responsibility for some clopen set containing~$\beta$ when~$\beta$ entered~$\BB^\s(n,\rho)$, and since no new computations are defined on that set while it is contained in~$\BB^\s(n,\rho)$ (by construction), $\BB^\s(n,\rho)$ had performed a test for~$\alpha$ when it did so.

If~$\beta \not \in \BB^\s_s(n,\rho)$, then let~$t < s$ be the largest stage such that~$\beta$ entered $\UU^\s$.  By the same reasoning as the previous lemma, $\beta$ cannot enter~$\UU^\tau$ for any~$\tau$ extending~$\s$ between stages~$t$ and~$s$.  So $\UU^\s$ took responsibility for some clopen set containing~$\beta$ at stage~$t$, and since no new computations have been defined since then,~$\UU^\s$ had performed a test for~$\alpha$ when it did so.

Thus all measure that enters~$\EE$ does so through one of two streams: the bin stream, and the~$\UU$ stream.  In each case, some object had taken responsibility for that bit of measure because of some successful test for some~$\alpha$, and the measure entering~$\EE$ indicates that~$\alpha \not \prec A$.  We charge that measure to the account of this object.  It thus suffices to total the accounts through the construction.

For $\UU^\s$'s, every entry to~$\EE_s$ which is charged to~$\UU^\s$ is a subset of~$\UU^\s_s$.  So such an entry has measure at most $2^{-k} = 4\delta_\s$ and is tested on~$I_\s$.  Since $b_s(I_\s)$ is bounded by~$k$, the total charge to the account is bounded by $k2^{-k}$. Since $k>5$, this is bounded by $2^{-(k/2)}$.

We have $\delta_{\s}\le 2^{-2m}\delta_{\s^-} \le 2^{-2m}/4$, where $2^{-m} \le \epsilon_\sigma$.  So $2^{-k} \le 2^{-2m}$ and thus the charge to the account is bounded by $2^{-(k/2)} \le 2^{-m} \le \epsilon_\s$.

For bins $\BB^\s(n,\rho)$, let $2^{-k} = \min\{ \delta_\s, 2^{-2n}\}$ be the bound on the size of the bin.  Measure claimed with certainty~$k'$ can go bad no more than~$k'$ many times, and each time it can be an amount of measure no more than~$2^{-k'+1}$.  This means that the total measure charged to~$\BB^\s(\rho,n)$ is at most 
\[ \sum_{k'\ge k} k'2^{-k'+1} = (4k+2)2^{-k}.\]
Again because $k>12$ this is bounded by $2^{-(k/2)}$. Now the total damage for the test $\seq{\VV^\s_n}_{n\ge n_\s}$ is bounded by the sum 
\[ \sum_{n\ge n_\s} \sum_{\tau\supseteq \s} \min\{2^{-n}, 2^{-k}\} \,\,\,\Cyl{\delta^\s_\tau = 2^{-2k}} \]
which is bounded by $6\cdot 2^{-n_\s}$, which is bounded by $\epsilon_\s$. 

Thus the total charge over all accounts is bounded by $\sum_\s 2\epsilon_\s \le 1/4$.
\end{proof}

\begin{cclaim}
At every stage~$s$, for any~$\s$, $\leb\left(\UU^{\s\conc\fin} \cup \bigcup_d \UU^{\s\conc d}\right)[s] \le \delta_\s$.
\end{cclaim}

\begin{proof}
\begin{align*}
\leb\left(\UU^{\s\conc\fin} \cup \bigcup_d \UU^{\s\conc d}\right)[s] &= 4\delta_{\s\conc\fin} + 4\cdot\sum_d \delta_{\s\conc d}\\
&\le 4\delta_\s\cdot 2^{-12} + 4\cdot\sum_d \delta_\s\cdot 2^{-(12 + d)}\\
&= 4\delta_\s\cdot 2^{-10} < \delta_\s.\qedhere
\end{align*}
\end{proof}

\begin{cclaim}\label{claim:small cover of U}
At every stage~$s$, for any~$\s$, $\leb\left(\bigcup_{n\ge n_\s} \VV_{n,s}^\s\right) \le \delta_\s$.
\end{cclaim}

\begin{proof}
Immediate from choice of~$n_\s$.
\end{proof}

\begin{cclaim}\label{claim:U full size}
At every stage~$s$ when a node~$\s$ is accessible, $\leb\bigl(\UU^\s \cap (\WW^\s \cup \EE)\bigr)[s] \le \delta_\s$.
\end{cclaim}

\begin{proof}
If~$\s = \seq{}$, this is Claim \ref{lem:total bad measure} together with the observation that~$F^\seq{}$ is empty, and thus so is~$\WW^\seq{}_s$.  For larger~$\s$, this is immediate from the action of~$\s^-$.
\end{proof}

\begin{cclaim}
At every stage~$s$ when a node~$\s$ is accessible, $\lambda(\UU^\s_s) = 4\delta_\s$.
\end{cclaim}

\begin{proof}
Immediate from construction and simultaneous induction with the next claim.
\end{proof}

\begin{cclaim}
The sets~$\YY^\tau$ can always be chosen as described in the action of~$\s$.
\end{cclaim}

\begin{proof}
\begin{align*}
&\leb\left(\WW^\s \cup \bigcup_{n > n_\s} \VV_{n}^\s \cup \bigcup_d \UU^{\s\conc d} \cup \UU^{\s\conc\fin} \cup (\UU^\s \cap \EE)\right)[s]\\
&\leq\leb\bigl(\UU^\s \cap (\WW^\s \cup \EE)\bigr)[s] + \leb\left(\bigcup_{n \ge n_\s}\VV_{n,s}^\s\right) + \leb\left(\bigcup_d \UU^{\s\conc d} \cup \UU^{\s\conc\fin}\right)[s]\\
&\le 3\cdot\delta_\s.
\end{align*}
Since~$\leb(\UU^\s_s) = 4\delta_\s$, there is at least~$\delta_\s$ available measure to draw from for~$\YY^\tau$.  Since~$4\delta_\tau < \delta_\s$, this is sufficient.
\end{proof}

\begin{cclaim}
If~$\s$ is on the true path, then~$\s$ is accessible infinitely often.
\end{cclaim}

\begin{proof}
By induction on~$|\s|$.  The base case is trivial.

For $|\s| = e + 1$, if $f(e) \in \w$, this is by definition.  If $f(e) = \fin$, then let~$n$ be least such that~$g^e(n)\diverge$, and let~$s > e$ be a stage such that $g\uhr{n}$ has converged by~$s$.  Then at every stage after~$s$ when $\s^- = f\uhr{e}$ is accessible, if~$\s^-$ reaches substage 3, it will choose $\tau = \s$.  The only reason for~$\s$ not to be accessible is if~$\s^-$ decides it wants to change~$\UU^\s$, but this can happen at most~$p(\s)$ many times.
\end{proof}

\begin{cclaim}
If~$\s$ is on the true path, then for every $\rho \in F^\s$,~$g^\rho$ is total.
\end{cclaim}

\begin{proof}
Immediate from the definition of the true path.
\end{proof}

\begin{cclaim}\label{claim:tot outcome exists}
Suppose $\s = (f\uhr{e})$ is on the true path and~$g^e$ is total.  Then there is a~$d$ such that~$\s\conc d$ is infinitely often accessible.
\end{cclaim}

\begin{proof}
By assumption~$A$ is strongly jump traceable, so let~$d'$ be least such that~$\bar{T}^{\s\conc d'}$ traces~$\Psi_\s^A$.  Fixing a stage~$s_0$, note that there are only finitely many tests on~$\Psi_\s^A$ begun by extensions of~$\s\conc d'$ at stage~$s_0$, and no new test will be begun until~$\s\conc d'$ is again accessible.  Since~$\bar{T}^{\s\conc d'}$ is a trace, eventually these finitely many tests will return.

Similarly, there are only finitely many tests on~$\Psi_{\s\conc d'}^A$ which are waiting to begin.  Since $g^e = g^{\s\conc d'}$ and~$g^\rho$ for all~$\rho \in F^{\s\conc d'}$ are total, eventually the appropriate intervals will be defined.

Note by construction that~$\s$ reaches substage 3 at least every other time it is accessible.  Thus there is some~$s_1 \geq s_0$ when~$\s$ reaches substage 3 and~$\s\conc d'$ is a valid choice for~$\tau$.  At this stage, $\tau = \s\conc d''$ for some $d'' \leq d'$.

By pigeon-hole, there is some least $d \leq d'$ which is infinitely often accessible.
\end{proof}

\begin{cclaim}\label{claim:intersection nonempty}
The intersection $\bigcap (\UU^{f\uhr{e}} - \WW^{f\uhr{e}} - \EE) \,\,\,\Cyl{e \in \w}$ is non-empty.
\end{cclaim}

\begin{proof}
By construction, $\UU^{f\uhr{e}} \supset \UU^{f\uhr{e'}}$ for $e' > e$, and $\UU^{f\uhr{e}} \cap \WW^{f\uhr{e}} \subseteq \WW^{f\uhr{e'}}$ for $e' > e$ by definition.  Thus $\seq{ \UU^{f\uhr{e}} - \WW^{f\uhr{e}} - \EE}$ is nested.  Further,~$\UU^{f\uhr{e}}$ is closed, and~$\WW^{f\uhr{e}}$ and~$\EE$ are open.  Thus $\UU^{f\uhr{e}} - \WW^{f\uhr{e}} - \EE$ is closed.  By Claims \ref{claim:small cover of U} and \ref{claim:U full size}, these are all non-empty, and so, by compactness of~$2^\w$, the intersection is non-empty.
\end{proof}

\begin{cclaim}
$X$ is Demuth random.
\end{cclaim}

\begin{proof}
By definition of the true path and our choice of enumeration, every quick clopen Demuth test occurs as~$\seq{\VV_n^{f\uhr{e}}}$ for some $e > 0$ with $f(e-1) \neq \fin$.  Then $(f\uhr{e}) \in F^{f\uhr{e+1}}$, so
\[
\WW^{f\uhr{e+1}} \supseteq \UU^{f\uhr{e+1}} \cap \bigcup \VV_n^{f\uhr{e}} \,\,\,\Cyl{n_{f\uhr{e}} \leq n}.
\]
Since $X \in \UU^{f\uhr{e+1}} - \WW^{f\uhr{e+1}}$,~$X$ passes the Demuth test~$\seq{\VV_n^{f\uhr{e}}}$.
\end{proof}

\begin{cclaim}
The definition of~$\Gamma$ is consistent.
\end{cclaim}

\begin{proof}
By construction, we only enumerate axioms during the second substage when we add measure to some~$\UU^\tau$.  In order for this measure to be added, a test of some~$\alpha$ with $|\alpha| \geq |\tau|$ must have been begun at some stage~$t$, and then successfully returned at stage $s \geq t$ when~$\tau^-$ was accessible.  By construction $\alpha \subset A_t$, and since this test has returned successfully, $\alpha \subset A_s$.  In particular, $A_t\uhr{|\tau|} = A_s\uhr{|\tau|}$.

The measure in~$\UU^\tau_{s+1}$, on which we define the new axiom, was chosen at stage~$t$ to be disjoint from~$\EE_t$.  Thus defining $\Gamma(\UU^\tau_{s+1}) = A_t\uhr{|\tau|}$ was consistent at stage~$t$.  By construction, $\tau^-$ was not accessible between stages~$t$ and~$s$, and thus no other computations were enumerated onto any of~$\UU^{\tau^-}_t$ between those stages, in particular none on~$\UU^\tau_{s+1}$.  Thus enumerating the computation $\Gamma(\UU^\tau_{s+1}) = A_t\uhr{|\tau|}$ is still consistent at stage~$s$.

The actual computation enumerated is $\Gamma(\UU^\tau_{s+1}) = A_s\uhr{|\tau|}$, but as observed above, $A_s\uhr{|\tau|} = A_t\uhr{|\tau|}$.
\end{proof}

\begin{cclaim}
$\Gamma^X = A$.
\end{cclaim}

\begin{proof}
Since $X \in \UU^{f\uhr{e}}$ for all~$e$, and $|\Gamma(\UU^{f\uhr{e}})| = e$, $\Gamma^X$ is total.  Since~$X \not \in \EE$, $\Gamma^X = A$.
\end{proof}

Thus~$X$ is the desired Demuth random set, completing the proof of Theorem \ref{thm:SJT Below Demuth}.

\section{Proof of Theorem \ref{thm:SJT Not Base for Demuth Randomness}} \label{sec:base}

We enumerate a c.e.\ set $A$. To ensure that $A$ is strongly jump-traceable, we meet the following requirements:
\begin{description}
	\item[$N_e$] if $h_e$ is an order function, then $J^A$ has an $h_e$-trace $\seq{T^e_x}_{x<\w}$.
\end{description}
Here $\seq{h_e}$ is an effective list of all partial computable functions whose domain is an initial segment of $\w$ and which are nondecreasing on their domain, and $J^A$ denotes a universal $A$-partial computable function. 

We will find the following approximation to the use function helpful: if $J^A(x)\, [s]$ converges, we define $j_s(x)$ to be the use of the computation~$J^A(x)\, [s]$.  Otherwise, we define $j_s(x) = 0$.

To ensure that~$A$ is not computable from an $A$-Demuth random set, we meet the following requirements:
\begin{description}
	\item[$P_e$] Every $X\in 2^\w$ such that $\Phi_e(X)=A$ fails an $A$-Demuth test $\seq{\UU_n}_{n<\w}$. 
\end{description}
This will be a single Demuth test shared by all~$P_e$.  To meet $P_e$, we meet subrequirements:
\begin{description}
	\item[$P_{e,m}$] If $\Phi_e(X)=A$ then there is some $n>m$ such that $X\in \UU_n$.  
\end{description}

To guess which functions $h_e$ are in fact order functions, we use a tree of strategies. To define the tree, we list the possible outcomes of each strategy (node on the tree). Let $\s$ be a node. If $\s$ works for $N_e$, then the possible outcomes of $\s$ are $\inff$ and $\fin$ (denoting whether or not~$h_e$ is an order). Otherwise, $\s$ has only one outcome. 

For a node $\s$, accessible at stage $s$ and of length $\le s$, we describe what actions $\s$ takes at stage $s$, and if $|\s|<s$, which outcome of $\s$ is next accessible.

\subsection{Strategy for $N_e$-requirements}

Let~$\ell_s(e)$ be the greatest~$n$ such that there is some $x \leq s$ such that $h_{e,s}(x)\converge = n$.

At stage~$s$,~$\s$ attends to all~$x$ such that~$h_e(x) < \ell_s(e)$: if $J^A(x)\converge = y\, [s]$ and $y \not \in T_{x,s}^\s$, then we enumerate~$y$ into~$T_x^\s$ and initialise all $\tau \supseteq \s\fin$.

Also, $\s$ aggregates restraint: for all $n<\ell_s(e)$, we let $R_s(\s,n)$ be the maximum of $j_s(x)$ where $h_{e}(x) \le n$.

Let $t<s$ be the last stage at which $\s\inff$ was accessible ($t=0$ if there is no such stage). If $\ell_s(e)> 2^{t+2}$ and $\ell_s(e) > 2^{n_t(\tau)}$ for every $\tau \supseteq \s\inff$ such that $n_t(\tau)$ is defined\footnote{$n_t(\tau)$ will be defined in the $P_{e,m}$-strategy}, let $\s\inff$ be accessible at stage $s$. Otherwise, $\s\fin$ is accessible at stage~$s$.

\subsection{A Basic Strategy for $P_{e,m}$}

When we see measure that appears to compute~$A$ (using $\Phi_e$), we have two possible ways in which we can satisfy~$P_{e,m}$: we can cover that measure with our test, or we can change~$A$.  We employ a combination of the two.

First, we choose an unclaimed test component~$\UU_n$ with $n > m$ and a large~$y$.  We keep~$y$ out of~$A$.  We study the open set
\[
\VV_s = \{X \in 2^\w \mid \Phi^X_{e,s} \supseteq A_s\uhr{y+1}\}.
\]
While $\leb(\VV_s) \leq 2^{-n}$, we can cover it with~$\UU_n$.  When~$\VV_s$ grows to be too large, we can enumerate~$y$ into~$A_{s+1}$.  Then all of~$\VV_s$ is wrong (it is in~$\EE$, to use the notation of the previous proof), so~$P_{e,m}$ need no longer concern itself with it.  We can then empty~$\UU_n$, choose a new~$y$, and start again.  This can happen at most~$2^n$ many times (since at least $2^{-n}$ measure goes bad each time it happens), so we have a computable bound on the number of times we empty~$\UU_n$.

This strategy is insufficient, however, because the strongly jump-traceable strategies~$N_{e'}$ act to ensure the trace at~$x$ by putting restraint on~$A$.  If a higher priority strategy places restraint that prevents~$y$ from entering~$A$, it will interfere with the $P_{e,m}$-strategy.

Our response is to modify the strategy slightly.  If~$y$ is restrained from entering~$A$, we empty~$\UU_n$, choose a new large~$y$, and start again.  Every $N_{e'}$-strategy will only impose restraint for~$x$ at most~$h_e(x)$ many times, so eventually this stops occurring.

It would seem that we have just constructed a strongly jump traceable c.e.\ set~$A$ which is not computable from a Demuth random, in contradiction with the previous theorem.  There is a complication, however, in the bound on the number of changes to~$\UU_n$; specifically, how many~$x$ are there with higher priority?

When a $P_{e,m}$-strategy is initialised, it chooses a test component~$\UU_n$ to work with.  This indicates that it will enumerate at most~$n$ many elements~$y$ into~$A$.  This then determines which $x$ are higher priority than~$P_{e,m}$; those pairs $(e', x)$ such that~$h_{e'}(x)$ is large enough to tolerate~$n$ many changes are lower priority, and the rest are higher.  Thus this choice of~$n$ would seem to indicate how many higher priority pairs there are.  However, we will not actually know how many such~$x$ there are until all the~$h_{e'}$ have grown sufficiently large.

To guess which functions $h_{e'}$ are in fact order functions, we use a tree of strategies.  The first time a $P_{e,m}$-strategy is accessible, the value~$n$ is chosen.  We will not let this strategy be accessible again until every~$h_{e'}$ which it guesses to be an order grows large relative to~$n$ (this is why we required $\ell_s(e) > n_t(\tau)$ in the $N_e$-strategy in order for~$\s\inff$ to be accessible).  So the second time the $P_{e,m}$-strategy is accessible, we know how many higher priority $(e', x)$ pairs there are, and thus what the bound on the number of changes to~$\UU_n$ is.  Unfortunately, this means that if a $P_{e,m}$-strategy is accessible precisely once, the computable bound we define will not be defined at~$n$.

Our strategy then is to define the bound on the number of changes to~$\UU_n$ to be~$0$ when the $P_{e,m}$-strategy is first accessible.  The second time the strategy is visited, we cause a change in~$A$ and redefine the bound to be whatever we now know it should be.  Because our redefinition accompanied a change in~$A$, the resulting function is $A$-computable.  Hence $\seq{\UU_n}$ will be an $A$-Demuth test.  Indeed, the only part of our $A$-test which requires the oracle is the bound on the number of changes.

\subsection{The Full Strategy for $P_{e,m}$}

$\s$ is associated with a \emph{test component} $n_s(\s)>m$, a \emph{coding marker} $x_s(\s)$ and a \emph{witness} $y_s(\s)$. These become undefined whenever~$\s$ is initialised.  Whenever~$\s$ changes the definitions of any of these or undefines them, all~$\tau \supset \s$ are initialised.
Let $s_0$ be the stage at which $\s$ was last initialised.  There are three cases.

\subsubsection*{Case 1} $n_s(\s)$ is not defined.

We set $n_s(\s)$ and $x_s(\s)$ to be large and pass to the next accessible node. 

\subsubsection*{Case 2} $n_s(\s)$ and $x = x_s(\s)$ are defined and $x \not \in A$

Let $b = 2^{s_0+2}$. If there is some $\beta\inf\subseteq \s$ such that $\beta$ works for some $N_d$\nbd requirement\footnote{\label{footnote}Since~$\s$ is accessible, we know that $\ell_s(d)$ is greater than~$b$ and~$k$, and thus that~$R_s(\beta,b)$ and~$R_s(\beta,k)$ are defined.}\addtocounter{footnote}{-1}\addtocounter{Hfootnote}{-1}, and $R_s(\beta,b)\ge x_s(\s)$, then we undefine both $x_s(\s)$ and $n_s(\s)$.
Otherwise, we enumerate $x$ into $A$.  Either way, we then pass to the next accessible node.

\subsubsection*{Case 3} $n = n_s(\s)$ and $x = x_s(\s)$ are defined and $x \in A$.

Let $k = 2^{n+2}$. If there is some $\beta\inf\subseteq \s$ such that $\beta$ works for some $N_d$\nbd requirement\footnotemark, and $R_s(\beta,k)\ge y_s(\s)$ (or if $y_s(\s)$ is not defined), then we set $y_s(\s)$ to be some large number and declare $\UU_n = \emptyset$.  

Let
\[ \VV_s(\s) = \left\{ X\in 2^\w\,:\, \Phi_{e,s}^X(x) \supseteq A_s\uhr{y+1} \right\}.\]
If $\lambda \VV_s(\s)\ge 2^{-n}$, we: 
\begin{itemize}
	\item enumerate $y$ into $A$, and declare $y_s(\s)$ to be undefined;
	\item declare $\UU_n = \emptyset$. 
\end{itemize}
Otherwise, we declare $\UU_n = \VV_s(\s)$.  We then pass to the next accessible node.

\subsection{Construction}

We build a tree of strategies by devoting each level to a single requirement.  Every strategy at level~$2e$ is devoted to the $N_e$-requirement, while every strategy at level $2\langle e, m\rangle + 1$ is devoted to the $P_{e,m}$-requirement.

At stage~$s$, we begin by letting the root be accessible and then proceed to let every accessible node~$\s$ with $|\s| < s$ act in order of length.  At the end of stage~$s$, for every~$\s$ with $|\s| < s$, we let~$x_{s+1}(\s) = x_s(\s)$, $y_{s+1}(\s) = y_s(\s)$ and $n_{s+1}(\s) = n_s(\s)$.

\subsection{Verification}

We perform the verification as a sequence of claims.

\begin{cclaim}\label{lem:bounding enumeration}
	Let $\s$ work for some $P_{e,m}$. Let $t>s$, and suppose that $n = n_s(\s) = n_t(\s)$. Then between stages $s$ and $t$, $\s$ enumerates at most $2^{n_s(\s)}$ many witnesses into~$A$. 
\end{cclaim}

\begin{proof}
	Let $s_0\le s$ be the stage at which the location $n= n_s(\s) = n_{s_0}(\s)$ was chosen. Let $s_1<s_2<\dots$ be the stages, after stage $s_0$, at which a new witness $y_i = y_{s_i}(\s)$ is chosen. Since each $y_i$ is chosen large, we have $y_1<y_2<\dots$. 
	
	Let $\VV_i = \VV_{s_{i+1}}(\s)$. We claim that if $n_{s_{i+1}}(\s)=n$ (so in particular, if ${s_{i+1}}<t$), then $\VV_i$ is disjoint from every $\VV_j$ for $j<i$. This is because for all $X\in \VV_j$ we have $\Phi_e^X(y_j)=0$, as $y_j\notin A_{s_{j+1}}$, but for all $X\in \VV_i$ we have $\Phi_e^X(y_j)=1$, as $y_j\in A_{s_{i+1}}$. 
	
	Since, for all $j$ such that $s_{j+1}$ is defined, we have $\lambda \VV_j \ge 2^{-n}$, we see that $s_{2^n+1}$ cannot exist.
\end{proof}

\begin{cclaim}\label{lem:negative bound 1}
Let~$\s$ work for some~$N_e$, and suppose that $m < \ell_s(e)$.  Then there are fewer than~$m$ many stages $s' \geq s$ at which some~$\tau \supseteq \s\inff$ enumerates an element into~$A$ below $R_{s'}(\s,m)$.
\end{cclaim}

\begin{proof}
Such elements come in two sorts: $x_{s'}(\tau)$ and $y_{s'}(\tau)$.  We count these separately.

By construction, in order for $x_{s'}(\tau) < R_{s'}(\s,m)$ to be enumerated into~$A$ at stage~$s'$, it must be that $2^{s_0+2} < m$, where~$s_0 < s'$ is the last stage at which~$\tau$ was initialised.  But since $|\tau| < s_0$ and the priority tree is at most binary branching, there are at most $2^{s_0}$ many strategies~$\tau$ which were initialised at stage~$s_0$.  Thus a bound on the number of such~$x_{s'}(\tau)$ is
\[
\sum_{2^{s_0+2} < m} 2^{s_0} < m/2.
\]

By construction, in order for~$y_{s'}(\tau) < R_{s'}(\s,m)$ to be enumerated into~$A$ at stage~$s'$, it must be that $2^{n+2} < m$, where $n = n_{s'}(\tau)$.  Since strategies always choose their~$n$ large, the same~$n$ never occurs more than once.  For a fixed~$n$, by Claim \ref{lem:bounding enumeration}, at most~$2^n$ many witnesses are enumerated.  Thus a bound on the number of such~$y_{s'}(\tau)$ is
\[
\sum_{2^{n+2} < m} 2^n < m/2.
\]

So there are fewer than~$m$ many such elements enumerated in total.  Since no element is enumerated more than once, there are fewer than~$m$ many such stages.
\end{proof}

\begin{cclaim}\label{lem:negative bound 2}
	Let $\sigma$ work for some $N_e$.  Let $t > s$, and suppose that $h_{e,s}(x)\converge < \ell_s(e)$, $\s$ is not initialised between stages $s$ and $t$.  Then between stages $s$ and $t$, at most $h_{e}(x)$ many elements are enumerated into $T_x^e$.
\end{cclaim}

\begin{proof}
Let $s_1 < s_2 < \dots$ be the stages between $s$ and $t$ at which~$\s$ enumerates an element into~$T_x^e$.  Then by construction, $J^A(x)\, [s_i] \neq J^A(x)\, [s_{i+1}]$, and so between stages~$i$ and~$i+1$ some accessible node~$\tau$ must have enumerated an element into~$A$ below $j_{s_i}(x)$.

By assumption, $\tau \not \subset \s$.  If~$\tau$ is to the left of~$\s$, then when~$\tau$ was accessible,~$\s$ would have been initialised, contrary to hypothesis.  If~$\tau$ is to the right of~$\s$ or $\tau \supseteq \s\fin$, then~$\tau$ was initialised at stage~$s_i$, and so the element~$x_{s'}(\tau)$ or~$y_{s'}(\tau)$ which was enumerated would have been chosen after stage~$s_i$, and thus would be larger than~$j_{s_i}(x)$.

So it must be that $\tau \supseteq \s\inff$.  But the number of stages at which this can happen is less than~$h_e(x)$ by Claim \ref{lem:negative bound 1}.  Thus there can be no~$s_{h_e(x)+1}$.
\end{proof}

\begin{cclaim}\label{lem:effectively open sets}
Let~$\s$ be working for some $P_{e,m}$-requirement.  Let $t > s$ be such that $n = n_s(\s) = n_t(\s)$.  Then if~$\UU_n$ is not declared empty between stages~$s$ and~$t$, $\UU_{n,s} \subseteq \UU_{n,t}$.
\end{cclaim}

\begin{proof}
By hypothesis, $\UU_{n,t} = \VV_t(\s)$, $\UU_{n,s} = \VV_s(\s)$, $y = y_s(\s) = y_t(\s)$ and $y \not \in A_t$.  If $\VV_s(\s) \not \subseteq \VV_t(\s)$, then $A_s\uhr{y+1} \neq A_t\uhr{y+1}$.  So some element less than~$y$ was enumerated into~$A$ by some accessible strategy~$\rho$ between stages~$s$ and~$t$.

If $\rho \subset \s$ or $\rho$ is to the left of~$\s$, then~$\s$ would have been initialised between stages~$s$ and~$t$ when~$\rho$ was accessible, contradicting $n_s(\s) = n_t(\s)$.

If $\rho \supset \s$ or $\rho$ is to the right of~$|s$, then~$\rho$ would have been initialised when~$\s$ chose~$y$, so any values chosen by~$\rho$ would be larger than~$y$.
\end{proof}

\begin{cclaim}\label{lem:bounding empty declarations}
There is an $A$-computable total function~$g(n)$ bounding the number of times~$U_n$ is declared empty.
\end{cclaim}

\begin{proof}
By construction, if~$n$ is not selected by some $P_{e,m}$-strategy by stage~$n$, it will never be selected, and thus~$g(n)$ can be set to 0.

Otherwise, let~$\s$ be the $P_{e,m}$-strategy which selects~$n$, let $s$ be the stage at which~$\s$ selects~$n$, and let $x = x_s(\s)$.  Note that by construction, if~$\s$ is accessible at stage $t > s$, $\ell_t(d) > 2^{s_0+2}$ and $\ell_t(d) > 2^{n+2}$ for all $N_d$-strategies $\beta\inff \subseteq \s$.

If $x \not \in A$, there are two possibilities: either $\s$ was never again accessible after stage~$s$, or $x$ and~$n$ were undefined before the next time~$\s$ was accessible after stage~$s$.  In both cases, $g(n) = 0$ suffices.

If $x \in A$, then~$\s$ was accessible at some stage $t > s$.  At this stage, for every $N_{d}$-strategy $\beta$ with $\beta\inff \subseteq \s$, we can compute $\#\{ x \mid h_{d}(x) \leq 2^{n+2}\}$.  By Claim \ref{lem:negative bound 2}, each such~$x$ can cause~$R(\beta,2^{n+2})$ to change at most~$2^{n+2}$ many times.

By construction, whenever~$\UU_n$ is declared empty, either a new~$y$ was chosen because the previous~$y$ was below some restraint, or because the previous~$y$ was enumerated into~$A$.  We can use the previous paragraph to bound the first number, and Claim \ref{lem:bounding enumeration} to bound the second.  Thus if~$x \in A$,
\[
g(n) = 2^n + \sum_{\tau\inf \subseteq \s} 2^{n+2}\cdot \#\{x \mid h_{d}(x) \leq 2^{n+2}\}
\]
suffices.
\end{proof}

\begin{cclaim}
$\seq{\UU_n}$ is an $A$-Demuth test.
\end{cclaim}

\begin{proof}
Claims \ref{lem:effectively open sets} and \ref{lem:bounding empty declarations}.
\end{proof}

Define the True Path in the usual fashion.

\begin{cclaim}\label{lem:finite initialisation}
Every strategy along the true path is initialised only finitely many times.
\end{cclaim}

\begin{proof}
Proof by induction.

Let~$\s$ be along the true path, and let $s_0$ be a stage such that for every $\tau\fin \subseteq \s$ with $\tau$ an $N_e$-strategy, $\ell_s(e)$ will never change after stage~$s_0$, and for every~$\rho \subset \s$, $\rho$ will never again enumerate an element into~$A$.  Then by construction,~$\s$ will never again be initialised.
\end{proof}

\begin{cclaim}
Every strategy along the true path guarantees its requirement.
\end{cclaim}

\begin{proof}
By construction, if~$h_e$ is an order,~$T_x^e$ traces~$J^A$.  By Claims \ref{lem:negative bound 2} and \ref{lem:finite initialisation},~$T_x^e$ is eventually smaller than the order~$h_e$, which suffices to meet the $N_e$-requirement.

Let~$\s$ be a $P_{e,m}$-strategy along the true path.  Let~$s_0$ be the final stage at which~$\s$ is initialised.  The next time~$\s$ is accessible after~$s_0$, we will choose an~$n$ and~$x$, and from then after never again consider Case 1.

By Claim \ref{lem:negative bound 2},~$R_s(\beta,b)$ will eventually stabilise for every~$\beta\inff \subseteq \s$.  Thus we will eventually enumerate some~$x$ into~$A$ and never again consider Case 2.

By Claim \ref{lem:negative bound 2} again,~$R_s(\beta,k)$ will eventually stabilise for every~$\beta\inff \subseteq \s$.  Thus we will eventually stop rechoosing~$y$ because of restraint.

By Claim \ref{lem:bounding enumeration}, we enumerate only finitely many of these~$y$ into~$A$.  After we have enumerated the last one, $\UU_n$ will cover all~$X$ which compute~$A$.
\end{proof}

This completes the proof of Theorem \ref{thm:SJT Not Base for Demuth Randomness}.


\end{document}